\documentclass[11pt,a4paper]{amsart}
\usepackage{epsfig}
\usepackage{amsmath}
\usepackage{amsthm} 
\usepackage{amssymb}
\usepackage{bm}
\usepackage{tikz}
\usepackage{PTSansNarrow}
\usepackage[T1]{fontenc}
\usetikzlibrary{matrix}
\usepackage{threeparttable}
\usepackage{graphicx}
\newtheorem{theorem}{Theorem}[section]
\newtheorem{lemma}[theorem]{Lemma}

\newtheorem{remark}[theorem]{Remark}
 \numberwithin{dummy}{section}
\newtheorem{algorithm}{Algorithm}
\newtheorem*{algorithm-PF}{Primal Formulation}
\newtheorem*{algorithm-PMF}{Primal-Mixed Formulation}
\newtheorem*{algorithm-DMF}{Dual-Mixed Formulation}
\newtheorem*{algorithm-PF-C}{Conforming Finite Element Method}
\newtheorem*{algorithm-PMF-C}{Conforming Primal-Mixed Finite Element Method}
\newtheorem*{algorithm-DMF-C}{Conforming Mixed Finite Element Method}
\newtheorem*{algorithm-hmfem}{Hybridized Mixed Finite Element Method}
\newtheorem*{algorithm-primal-wgfem}{Primal WG-FEM}
\newtheorem*{algorithm-primalmixed-wgfem}{Primal-Mixed WG-FEM}
\newtheorem*{algorithm-mixed-wgfem}{Mixed WG-FEM}
\newtheorem*{algorithm-hmwgfem}{Hybridized Mixed WG-FEM}
\newtheorem*{algorithm-hdg}{Hybridizable Discontinuous Galerkin}
\numberwithin{equation}{section}

\newcommand{\bq}{{\bm q}}
\newcommand{\bn}{{\bm n}}
\newcommand{\bx}{{\bm x}}
\newcommand{\bu}{{\bm u}}
\newcommand{\bv}{{\bm v}}
\newcommand{\bE}{{\textbf{e}}}

\newcommand{\bV}{{\boldsymbol V}}
\newcommand{\bW}{{\boldsymbol W}}
\newcommand\subsetsim{\mathrel{%
  \ooalign{\raise0.2ex\hbox{$\subset$}\cr\hidewidth\raise-0.8ex\hbox{\scalebox{0.9}{$\sim$}}\hidewidth\cr}}}
\def\leqC{\lesssim}

\def\T{{\mathcal T}}

\def\bn{{\boldsymbol n}}
\def\bq{{\boldsymbol q}}

\def\boldeta{{\boldsymbol\eta}}
\def\bpsi{{\boldsymbol\psi}}

\def\bvarphi{{\boldsymbol\varphi}}

\def\bbQ{\mathbb{Q}}

\def\mathcalQ{{\mathcal Q}}
\def\Wspace{{\mathbb{W}_\varepsilon(\Omega)}}

\def\boldeta{\bm{\eta}}
\newcommand{\pT}{{\partial T}}
\newcommand{\curl}{{\nabla\times}}

\def\3bar{{|\hspace{-.02in}|\hspace{-.02in}|}}
\renewcommand{\ldots}{\dotsc}
\newcommand{\vertiii}[1]{{\left\vert\kern-0.25ex\left\vert\kern-0.25ex\left\vert #1
\right\vert\kern-0.25ex\right\vert\kern-0.25ex\right\vert}}
\theoremstyle{definition}

\newtheorem{example}[theorem]{Example}

\theoremstyle{remark}

\numberwithin{equation}{section}

\usepackage{graphicx}
\usepackage{caption,subcaption}
\makeatletter
\newcommand{\tnorm}{\@ifstar\@tnorms\@tnorm}
\newcommand{\@tnorm}[2][]{%
\mathopen{#1|\mkern-2.5mu#1|\mkern-2.5mu#1|}
#2
\mathclose{#1|\mkern-2.5mu#1|\mkern-2.5mu#1|}
}
\makeatother

\usepackage{hyperref}

\makeatletter
\@namedef{subjclassname@2020}{\textup{2020} Mathematics Subject Classification}
\makeatother

\begin{document}
\title{An $L^p$- Primal-Dual Weak Galerkin method for div-curl Systems}
 
\author{Waixiang Cao} \address{School of Mathematical Science, Beijing Normal University, Beijing 100875}
\email{caowx@bnu.edu.cn} \thanks{ The research of Waixiang Cao was partially supported by National Science Foundation of China grant No. 11871106} 
\author{Chunmei Wang}
\address{Department of Mathematics, University of Florida, Gainesville, FL 32611} \email{chunmei.wang@ufl.edu} \thanks{The research of Chunmei Wang was partially supported by National Science Foundation Award DMS-2136380.}
\author{Junping Wang}
\address{Division of Mathematical Sciences, National Science
Foundation, Alexandria, VA 22314}
\email{jwang@nsf.gov}
\thanks{The research of Junping Wang was supported by the NSF IR/D program, while working at
National Science Foundation. However, any opinion, finding, and
conclusions or recommendations expressed in this material are those
of the author and do not necessarily reflect the views of the
National Science Foundation.}

\begin{abstract} 
This paper presents a new $L^p$-primal-dual weak Galerkin (PDWG) finite element method for the div-curl system with the normal boundary condition for $p>1$. Two crucial features for the proposed $L^p$-PDWG finite element scheme are as follows: (1) it offers an accurate and reliable numerical solution to the div-curl system under the low $W^{\alpha, p}$-regularity ($\alpha>0$) assumption for the exact solution; (2) it offers an effective approximation of the normal harmonic vector fields on domains with complex topology.  An optimal order error estimate is established in the $L^q$-norm for the primal variable where $\frac{1}{p}+\frac{1}{q}=1$.
A series of  numerical experiments are  presented to demonstrate the performance of the proposed $L^p$-PDWG algorithm. 
\end{abstract}



\pagestyle{myheadings}

\maketitle

\section{Introduction}\label{Section:Introduction}
In this paper we shall develop a new $L^p$-primal-dual weak Galerkin (PDWG)  methods for the div-curl system with the normal boundary condition. To this end, we consider the model problem: Find a vector field $\bm{u}=\bm{u}(\bx)$ such that
\begin{subequations}\label{EQ:div-curl}
\begin{align}
\nabla\cdot(\varepsilon \bm{u})&=f,\qquad {\rm in}\  \Omega, \label{EQ:div-curl:div-eq}\\
\nabla\times \bm{u}&= \bm{g},\qquad {\rm in}\  \Omega, \label{EQ:div-curl:curl-eq}\\
\varepsilon \bm{u} \cdot \bm{n}&=\phi_1, \qquad {\rm on}\ \Gamma, \label{EQ:div-curl:normalBC}
\end{align}
\end{subequations}
where $\Omega\subset {\mathbb R}^3$ is an open, bounded and connected polyhedral domain, and $\Gamma=\partial \Omega$ is the boundary of $\Omega$. Assume that $\Gamma$ is the union of a finite number of disjoint surfaces $\Gamma=\bigcup_{i=0}^L\Gamma_i$ with $\Gamma_0$ being the exterior boundary of $\Omega$ and $\Gamma_i \ (i=1,  \cdots, L)$ being the other connected components with finite surface areas. Note that $L$ is equal to the number of holes in the
domain $\Omega$ geometrically which is known as the second Betti number of $\Omega$ or the dimension of the second de Rham cohomology group of $\Omega$. Assume the coefficient matrix $\varepsilon= \{\varepsilon_{ij}(\bx)\}_{3\times 3}$ is symmetric and uniformly positive definite in $\Omega$ with   $\varepsilon_{ij}$ ($i, j=1,2,3$) being in $L^{\infty}(\Omega)$. 


The solution uniqueness for the div-curl system \eqref{EQ:div-curl:div-eq}-\eqref{EQ:div-curl:normalBC} depends on the topology of the domain $\Omega$. It is well-known that the solution uniqueness holds true for simply connected $\Omega$, while the solution is unique up to a normal $\varepsilon$-harmonic function in $\mathbb{W}_0^{\varepsilon n,p}(\Omega)$ defined in (\ref{harmo}) for the case that the domain $\Omega$ is not simply connected. The dimension of $\mathbb{W}_0^{\varepsilon n,p}(\Omega)$ is the first Betti number of $\Omega$ which is 
the rank of the first homology group of $\Omega$. 

The div-curl system (\ref{EQ:div-curl:div-eq})-(\ref{EQ:div-curl:curl-eq}) arises in many applications in science and engineering such as electromagnetic fields and fluid mechanics. Computational electro-magnetics plays an important role in many areas 
such as radar, satellite, antenna design, waveguides, optical fibers, medical imaging and design of invisible cloaking devices \cite{r9}. In linear magnetic fields, the function $f(\bx)$ vanishes, $\bu$ represents the magnetic field intensity and $\varepsilon(\bx)$ is the inverse of the magnetic permeability tensor. In fluid mechanics fields, the coefficient matrix $\varepsilon(\bx)$ is diagonal with diagonal entries being the local mass density.  In electrostatics fields, $\varepsilon(\bx)$ is the permittivity matrix. 

There have been several numerical methods  proposed and analyzed for the div-curl system (\ref{EQ:div-curl:div-eq})-(\ref{EQ:div-curl:curl-eq}). A covolume method was developed  by the employment of the Voronoi-Delaunay mesh pairs in three dimensional space \cite{r17}. \cite{r3} developed a least-squares finite element method for two types of boundary value problems.  The least-squares method was proposed  in \cite{r2} for the div-curl problem based on discontinuous elements on nonconvex polyhedral domains. 
 In \cite{3}, a classical numerical method was introduced for solving the magnetostatic problem by employing a scalar or vector potential. The control volume method \cite{15}  was proposed directly for planar div-curl problems. \cite{9} proposed a discrete duality finite volume method for div-curl problems on almost arbitrary polygonal meshes. 
A mixed finite element method was introduced in  \cite{r8} for three dimensional axisymmetric div-curl systems through a dimension reduction technique based on the cylindrical coordinates in simply connected and axisymmetric domains. The mimetic finite difference scheme \cite{5, r11}  was introduced for the magneto-static problems on general polyhedral partitions. The numerical algorithm \cite{17} was designed to construct a finite element basis for the first de Rham cohomology group of the computational domain, which was further used for a numerical approximation of the magnetostatic problem.       \cite{r23}  proposed a weak Galerkin finite element method for the div-curl system with either normal or tangential boundary conditions.  Another weak Galerkin scheme was introduced in \cite{li} by using a least-squares approach for the div-curl problem.  \cite{wl, divcurl} developed  primal-dual weak Galerkin finite element methods for the div-curl system with tangential boundary condition and normal boundary condition respectively and proved that the schemes work well for the exact solution with low-regularity assumptions.

There are two main challenges in the approximation of the div-curl system (\ref{EQ:div-curl:div-eq})-(\ref{EQ:div-curl:normalBC}): (1) the low-regularity of the exact solution $\bu$  limiting the stability and accuracy of the numerical solutions, and (2) the non-uniqueness of the solution $\bu$ on domains with complex topology. The later one can be relaxed to certain extent by seeking a particular solution orthogonal to the space of normal $\varepsilon$-harmonic vector space $\mathbb{W}_0^{\varepsilon n,p}(\Omega)$, but with an immediate obstacle lying in the determination of the space $\mathbb{W}_0^{\varepsilon n,p}(\Omega)$ or an effective approximation of this space. To address these challenges, we shall devise a new $L^p$ primal-dual weak Galerkin (PDWG) scheme for (\ref{EQ:div-curl:div-eq})-(\ref{EQ:div-curl:normalBC}) by following the framework developed in \cite{divcurl}.   It should be noted that the $L^p$-PDWG framework was originated in \cite{cao}   for  convection-diffusion equations. 
Our $L^p$-PDWG numerical method for (\ref{EQ:div-curl:div-eq})-(\ref{EQ:div-curl:normalBC}) has two prominent features over the existing numerical methods: (1)  it offers an effective approximation for the normal $\varepsilon$-harmonic vector space $\mathbb{W}_0^{\varepsilon n, p}  (\Omega)$ regardless of the topology of the domain $\Omega$; and (2) it provides an accurate and reliable numerical solution for the div-curl system (\ref{EQ:div-curl:div-eq})-(\ref{EQ:div-curl:normalBC}) with low $W^{\alpha, p}$-regularity ($\alpha>0$) assumption for the exact solution $\bu$. 

The paper is organized as follows. In Section \ref{Section:2}, we introduce the notation and  derive the weak formulation for the div-curl system \eqref{EQ:div-curl:div-eq}-\eqref{EQ:div-curl:normalBC}.
In Section \ref{Section:4}  a $L^p$-PDWG algorithm 
for both the div-curl problem and  the discrete normal $\varepsilon$-harmonic vector fields
 is proposed.  
  The solution existence and uniqueness for the $L^p$-PDWG scheme is discussed in Section \ref{Section:5}.  
  The convergence theory for the $L^p$-PDWG approximation is established in Section \ref{Section:7}. Finally,      several test examples are demonstrated to illustrate the performance of the $L^p$-PDWG algorithm in  Section \ref{Section:8}. 

\section{Weak formulations}\label{Section:2}
\subsection{Notations}
  We follow the usual notations for Sobolev spaces and norms \cite{ciarlet, girault-raviart}. Let   $D \subset {\mathbb R}^3$ be an open bounded domain  with Lipschitz continuous boundary. 
Denote by $W^{div_\varepsilon, p}(D)$  the closed subspace of $[L^p(D)]^3$ such that $\nabla\cdot(\varepsilon\bv)\in L^p(D)$. Denote $W^{div_\varepsilon, p}(D)$ by $W^{div, p}(D)$ when $\varepsilon=I$. Analogously, we use $W^{curl, p}(D)$ to denote the closed subspace of $[L^p(D)]^3$ so that $\nabla\times\bv\in [L^p(D)]^3$. Denote by $W_0^{curl, p}(D)$ the closed subspace with vanishing tangential boundary values, i.e., 
$$W_0^{curl, p}(D):=\{\bv\in W^{curl, p}(D),\ \bv\times\bn =0 \mbox{ on } \partial D\}.$$
Denote by $\langle\cdot,\cdot\rangle_{\Gamma_i}$ the inner product in $L^2(\Gamma_i)$.
We introduce the following Sobolev space
\begin{equation*}\label{EQ:NewSSpace}
\Wspace=\{\bv\in W_0^{curl, p}(\Omega)\cap W^{div_\varepsilon, p}(\Omega), \ \nabla\cdot(\varepsilon\bv)=0,\ \langle \varepsilon\bv\cdot\bn_i, 1\rangle_{\Gamma_i}=0,\ i=1,\cdots, L\}.
\end{equation*}
 
A vector field $\bv\in [L^p(\Omega)]^3$ is defined to be $\varepsilon$-harmonic in $\Omega$ if it is $\varepsilon$-solenoidal and irrotational in $\Omega$.   Denoted by $\mathbb{W}_0^{\varepsilon n,p}(\Omega)$ the space of normal $\varepsilon$-harmonic vector fields that consists of all $\varepsilon$-harmonic vector fields satisfying vanishing  normal boundary condition, i.e.,
\begin{equation}\label{harmo}
\mathbb{W}_0^{\varepsilon n,p}(\Omega)=\{\bv\in [L^p(\Omega)]^3: \ \curl\bv=0,\
\nabla\cdot(\varepsilon \bv)=0,\ \varepsilon\bv\cdot \bn = 0 \mbox{ on } \Gamma\}.
\end{equation}
 Denote $\mathbb{W}_0^{\varepsilon n,p}(\Omega)$ by $\mathbb{W}_0^{n,p}(\Omega)$ for  $\varepsilon=I$. Similarly, denoted by $\mathbb{W}_0^{\varepsilon \tau,p}(\Omega)$ the space of tangential $\varepsilon$-harmonic vector fields that consists of all $\varepsilon$-harmonic vector fields satisfying vanishing tangential boundary condition, i.e.,
$$
\mathbb{W}_0^{\varepsilon \tau,p}(\Omega)=\{\bv\in [L^p(\Omega)]^3: \ \curl\bv=0,\
\nabla\cdot(\varepsilon \bv)=0,\ \bv\times\bn = 0 \mbox{ on } \Gamma\}.
$$

\subsection{A Weak Formulation}
Testing  \eqref{EQ:div-curl:div-eq} by any $\varphi\in W^{1, p}(\Omega)$ and using the normal boundary condition (\ref{EQ:div-curl:normalBC}) yields
\begin{equation}\label{EQ:variational-form-1}
(\bm{u}, \varepsilon \nabla \varphi)  = \langle \phi_1, \varphi\rangle - (f,\varphi), \qquad\forall \varphi\in W^{1, p}(\Omega).
\end{equation}
Testing  \eqref{EQ:div-curl:curl-eq} by any $\bm{w}\in W_0^{curl, p}(\Omega)$  gives
\begin{eqnarray}\label{EQ:variational-form-2}
(\bm{u}, \nabla \times \bm{w}) = (\bm{g},\bm{w}), \qquad \forall \bm{w}\in W_0^{curl, p}(\Omega).
\end{eqnarray}
Combining with the equations \eqref{EQ:variational-form-1} and \eqref{EQ:variational-form-2}, we obtain a weak solution   $\bm{u}\in [L^q(\Omega)]^3$ ($\frac{1}{p}+\frac{1}{q}=1$) of the div-curl system with normal boundary condition  (\ref{EQ:div-curl:div-eq})-(\ref{EQ:div-curl:normalBC})  satisfying 
\begin{eqnarray}\label{EQ:variational-form}
(\bm{u}, \varepsilon \nabla \varphi + \nabla \times \bpsi ) = (\bm{g},\bpsi)-(f,\varphi) + \langle \phi_1, \varphi\rangle,
\end{eqnarray}
for all $\varphi \in  W^{1, p}(\Omega)$ and $\bpsi \in  W_0^{curl, p}(\Omega)$.

As discussed in \cite{divcurl}, the solution to the variational problem \eqref{EQ:variational-form} is  non-unique in general.  The
 homogeneous version of \eqref{EQ:variational-form} is to seek a  $\bu\in [L^q(\Omega)]^3 (\frac{1}{p}+\frac{1}{q}=1)$  satisfying
\begin{eqnarray}\label{EQ:variational-form-homo}
(\bm{u}, \varepsilon \nabla \varphi + \nabla \times \bpsi) = 0\qquad \forall \varphi\in W^{1, p}(\Omega),\ \forall \bpsi\in  W_0^{curl, p}(\Omega). 
\end{eqnarray}
Note that the solution could be any $\varepsilon$-harmonic function in $\mathbb{W}_0^{\varepsilon n,p}(\Omega)$ which is non-unique  provided that the $\varepsilon$-harmonic space $\mathbb{W}_0^{\varepsilon n,p}(\Omega)$ has a positive dimension. The solution to the div-curl system (\ref{EQ:div-curl:div-eq})-(\ref{EQ:div-curl:normalBC}) is unique provided that the solution is  $\varepsilon$-weighted $L^2$ orthogonal to $\mathbb{W}_0^{\varepsilon n,p}(\Omega)$.

\subsection{An Extended Weak Formulation}

In this subsection, we slightly modify the  
 weak formulation \eqref{EQ:variational-form-homo} to ensure that  the solution to the homogeneous version of \eqref{EQ:variational-form} is unique.
 
We first denote by $W_{0c}^{1,p}(\Omega)$ the subspace of   $W^{1,p}(\Omega)$ with vanishing value on $\Gamma_0$ and constant values on other connected components of the boundary; i.e.,
\begin{equation*}\label{EQ:Nov-11-2014:H0c}
W_{0c}^{1,p}(\Omega)=\{\phi\in W^{1,p}(\Omega):\ \phi|_{\Gamma_0}=0, \
\phi|_{\Gamma_i}=\alpha_i, \ i=1,\ldots, L\}.
\end{equation*}
Define the following bilinear form:
\begin{equation}\label{EQ:div-curl:nbvp:bform}
B(\bu,s; \varphi,\bpsi): = (\bu, \varepsilon \nabla \varphi + \nabla \times \bpsi) + (\bpsi, \varepsilon\nabla s). 
\end{equation}

  Now the extended weak formulation for   the div-curl system  (\ref{EQ:div-curl:div-eq})-(\ref{EQ:div-curl:normalBC}) seeks $(\bu, s)\in [L^q(\Omega)]^3\times W_{0c}^{1,p}(\Omega)$ satisfying
\begin{equation}\label{EQ:weakform-4-nbvp-divcurl}
B(\bu, s; \varphi,\bpsi) = F(\varphi,\bpsi),\qquad \forall \varphi\in W^{1,p}(\Omega), \forall \bpsi\in W_0^{curl,p}(\Omega), 
\end{equation}
  where 
\begin{equation}\label{ff:1}
F(\varphi,\bpsi)= (\bm{g},\bpsi)-(f,\varphi) + \langle \phi_1, \varphi\rangle.
\end{equation}
The homogeneous dual problem of \eqref{EQ:weakform-4-nbvp-divcurl} seeks $(\lambda, \bq ) \in W^{1,p}(\Omega)/\mathbb{R} \times W_0^{curl,p}(\Omega)$ such that
\begin{equation}\label{EQ:homo-dual-problem}
B(\bv, r; \lambda, \bq)=0,\qquad \forall \bv\in [L^q(\Omega)]^3,\ \forall r\in W_{0c}^{1,p}(\Omega).
\end{equation}
 It has been proved in \cite{divcurl} that the solution to the homogeneous dual problem \eqref{EQ:homo-dual-problem} is unique. 


 
\section{$L^p$-PDWG Scheme}\label{Section:4}
To design a $L^p$-PDWG scheme for the div-curl system  (\ref{EQ:div-curl:div-eq})-(\ref{EQ:div-curl:normalBC}), 
we first briefly review the definitions of discrete weak gradient and discrete weak curl \cite{divcurl} and then introduce some finite element spaces, which shall be used 
in our later algorithm.

Denote by ${\mathcal T}_h$  a finite element partition of the domain $\Omega$ that consists of
shape-regular polyhedra  \cite{wy3655}. Denote by ${\mathcal E}_h$ and ${\mathcal E}_h^0={\mathcal E}_h \setminus
\partial\Omega$  the set of all faces   and the set of all interior faces in ${\mathcal T}_h$ respectively. Let $h_T$ be the diameter of the element $T\in \T_h$ and $h=\max_{T\in {\mathcal T}_h}h_T$ be the meshsize of the partition $\T_h$.

  Let $T\in {\mathcal T}_h$ be a polyhedral domain with boundary $\partial T$. We define the space of scalar-valued weak functions on $T$ as follows
\begin{equation*}
W(T) =\{v = \{v_0, v_b\} : v_0 \in L^p(T), v_b \in L^p(\partial T)\},
\end{equation*}
where $v_0$ and $v_b$ represent  the values of $v$ in the interior and on the boundary  of $T$ respectively. Similarly, the space of vector-valued weak functions on $T$ is defined by
$$
V(T) =\{\bm{v} = \{\bm{v}_0, \bm{v}_b\} : \bm{v}_0 \in [L^p(T)]^3, \bm{v}_b \in [L^p(\partial T)]^3\}.
$$

Let $P_j(T)$ be   the polynomial  space on $T$ with total degree no more than $j$. Denote by $\bm{n}$ an unit outward normal direction  on $\partial T$. For any $v\in W(T)$, the discrete weak gradient  $\nabla_{w,j,T} v$  is defined as the unique vector-valued polynomial in $[P_j(T)]^3$ such that
\begin{equation}\label{EQ:dis_WeakGradient}
(\nabla_{w,j,T} v, \bm{\varphi})_T = -(v_0,\nabla \cdot \bm{\varphi})_T + \langle v_b , \bm{\varphi}\cdot \bn \rangle_{\partial T},\quad \forall \; \bm{\varphi} \in [P_j(T)]^{3}.
\end{equation}
Similarly, for any $\bv \in V (T)$, the discrete weak curl  $\nabla_{w,j,T}\times \bv$  is defined as the unique vector-valued polynomial in $[P_j(T)]^3$ such that
\begin{equation}\label{EQ:dis_WeakCurl}
(\nabla_{w,j,T} \times \bv, \bvarphi)_T = (\bv_0,\nabla \times \bvarphi)_T - \langle\bv_b \times \bn, \bvarphi\rangle_{\partial T}, \quad \forall \; \bm{\varphi} \in [P_j(T)]^3.
\end{equation}

For a given non-negative integer $k$, the finite element spaces are defined as follows 
\begin{equation*}
\begin{split}
\bV_h=&\{\bm{v}:\ \bm{v}|_T\in [P_k(T)]^3, \forall T\in\T_h \},\\
S_h=&\{ \{s_0,s_b\}: \ s_0|_T\in P_k(T), s_b|_\pT\in P_k(\pT), \forall T\in\T_h, s_b|_{\Gamma_0}=0, s_b|_{\Gamma_i}   \mbox{ is a constant} \},\\
M_h=&\{ \{\varphi_0,\varphi_b\}: \ \varphi_0|_T\in P_k(T), \varphi_b|_\pT\in P_k(\pT), \forall T\in\T_h,\ \int_\Omega \varphi_0 = 0\},\\
\bW_h=&\{\bm{\psi}=\{\bm{\psi}_0,\bm{\psi}_{b} \}: \ \bm{\psi}_0|_T\in [P_k(T)]^3, \bm{\psi}_{b}|_\pT\in G_k(\pT), \forall T\in\T_h, \bm{\psi}_{b}|_{\Gamma}=0\},
\end{split}
\end{equation*}
 where $G_k(\pT):=[P_k(\tau)]^3\times\bn_\tau$ is the space of polynomials of degree $k$ in the tangent space of $\pT$, 
and  $\bm{n}_{\tau}$ denotes the unit outward normal  vector  on $\tau$ with  $\tau\in \pT$.

For simplicity of notation and without confusion, for any $\sigma\in S_h$ or $\sigma\in M_h$, denote by  $\nabla_{w}\sigma$  the discrete weak gradient $\nabla _{w, k, T}\sigma$ computed  by  (\ref{EQ:dis_WeakGradient}) on  $T$, i.e.,
$$
(\nabla _{w} \sigma)|_T=\nabla _{w,k,T}(\sigma|_T), \qquad
\forall\sigma\in S_h \ \text{or}\ \sigma\in M_h.
$$
Similarly, for any $\bq\in \bW_h$, denote by $\nabla_{w} \times\bq$ the discrete weak
curl $\nabla_{w, k, T} \times \bq$ computed by (\ref{EQ:dis_WeakCurl}) on $T$, i.e.,
$$
(\nabla _{w} \times \bq)|_T=\nabla _{w, k, T} \times (\bq|_T), \qquad \forall\bq\in \bW_h.
$$

With the discrete weak gradient and discrete weak curl, an approximation of the bilinear form $B(\cdot; \cdot)$ is thus given by 
\begin{equation}\label{EQ:Bh}
B_h(\bm{v}, r; \varphi,\bpsi) = (v, \varepsilon \nabla_w \varphi+ \nabla_w \times \bpsi) + (\bpsi_0, \varepsilon\nabla_w r), 
\forall(\bm{v}, r,\varphi,\bpsi)\in \bV_h\times S_h\times M_h\times\bW_h.
\end{equation}

Now we are ready to present   the $L^p$-PDWG finite element method for the div-curl system  (\ref{EQ:div-curl:div-eq})-(\ref{EQ:div-curl:normalBC}).

\begin{algorithm}[$L^p$-PDWG Algorithm] The  $L^p$-PDWG finite element method for the div-curl system  (\ref{EQ:div-curl:div-eq})-(\ref{EQ:div-curl:normalBC})
seeks a $\bu_h\in \bV_h$, together with three auxiliary variables $s_h \in S_h$, $\lambda_h\in M_h$,  $\bq_h \in \bW_h$, such that
\begin{equation}\label{EQ:PDWG-3d:01}
\left\{
\begin{array}{rl}
s_1(\lambda_h, \bq_h;\varphi,\bpsi) + B_h(\bu_h, s_h; \varphi,\bpsi)&= F(\varphi,\bpsi),\quad \forall  \varphi\in M_h,\ \bpsi \in \bW_h, \\
-s_2(s_h, r)+B_h(\bm{v}, r; \lambda_h, \bq_h) & = 0, \qquad\qquad \forall \bm{v}\in \bV_h,\ r\in S_h.
\end{array}
\right.
\end{equation}
Here $F(\cdot,\cdot)$ is given in \eqref{ff:1}, and 
the $L^p$ stabilizer $s_1$  is defined  by  
\begin{equation*}\label{stab3}
\begin{split}
&s_1(\lambda_h, \bq_h;\varphi, \bm{\psi}) = \rho_1  \sum_{T\in \T_h}  \int_{\partial T} h_T^{1-p}|\lambda_0 -\lambda_b|^{p-1}sgn(\lambda_0 -\lambda_b)(\varphi_0 - \varphi_b)ds  \\
                                  &+ \rho_2 \sum_{T\in \T_h}  h_T^{1-p}  \int_{\partial T}|\bm{q}_0 \times \bm{n} -\bm{q}_{b} \times \bm{n}|^{p-1}sgn(\bm{q}_0 \times \bm{n} -\bm{q}_{b} \times \bm{n})(\bm{\psi}_0 \times \bm{n} -\bm{\psi}_{b} \times \bm{n})ds, 
\end{split}
\end{equation*}
and the $L^q$ stabilizer $s_2$ is defined accordingly in the space $M_h$ as follows
\begin{eqnarray*}\label{stab4}
s_2(s_h; r)=\rho_3 \sum_{T\in \T_h}h_T^{1-q}\int_{\partial T} |s_0 -s_b|^{q-1}sgn(s_0 -s_b)( r_0 - r_b)ds,
\end{eqnarray*}
where $p>1$, $q>1$ such that $\frac{1}{p}+\frac{1}{q}=1$, $\rho_i >0$ for $i=1, 2, 3$ are parameters with values at user's discretion.
\end{algorithm}

The above $L^p$-PDWG scheme \eqref{EQ:PDWG-3d:01} also offers an approximation of the normal $\varepsilon$-harmonic vector fields
$\mathbb{W}_0^{\varepsilon n,p}$.  Our later theoretical result (see Theorem \ref{THM:ErrorEstimate4uh})  demonstrates that 
the difference $\boldeta_h=\mathcalQ_h \bu - \bu_h$ is sufficiently close to a true normal $\varepsilon$-harmonic vector field $\boldeta$. 
Here $\mathcalQ_h$ denote the $L^2$ projection operator onto the finite element space $\bV_h$, and $\bu_h$ is the solution of \eqref{EQ:PDWG-3d:01} for the div-curl system \eqref{EQ:div-curl:div-eq}-\eqref{EQ:div-curl:normalBC}. Consequently,  
 a vector field $\boldeta_h\in \bV_h$ is said to be a 
 discrete normal $\varepsilon$-harmonic function  if there exists a vector field $\bu\in W^{div_\varepsilon, p}(\Omega)\cap W^{curl, p}(\Omega)$ satisfying
$\boldeta_h=\mathcalQ_h\bu - \bu_h.$


\section{Solution Existence and Uniqueness}\label{Section:5}
This section is dedicated to the study of solution existence and uniqueness of the $L^p$-PDWG scheme \eqref{EQ:PDWG-3d:01}. 
For simplicity, we assume that $\varepsilon$ is piecewise constant with respect to the partition $\T_h$ respectively. Note that all the results can be generalized to piecewise smooth $\varepsilon$ without any difficulty.

We define the following two semi-norms; i.e.,
\begin{equation}\label{norm}
\3bar (\lambda_h, \bq_h) \3bar =\Big(s_1(\lambda_h, \bq_h;\lambda_h, \bq_h)\Big)^{\frac{1}{p}},\ \  \lambda_h\in M_h,  \ \bq_h \in \bW_h,
\end{equation}
\begin{equation}\label{norm2}
\3bar s_h \3bar =\Big(s_2(s_h; s_h)\Big)^{\frac{1}{q}}, \ \  s_h \in S_h.
\end{equation}
 
Let $Q_h$ be the projection operator onto the weak finite element space $S_h$ or $M_h$ such that
$$
(Q_h w)|_T = \{Q_0 w|_T,Q_b w|_{\pT}\},
$$
where $Q_0$  and $Q_b$ are the $L^2$ projection operators onto $P_k(T)$ and $P_k(\tau)$ on each face $\tau \in \partial T$. 
Similarly, denote by $\bbQ_0$, $\bbQ_b$ and  $\bbQ_h$  the $L^2$ projection operators onto $[P_k(T)]^3$, $G_k(\tau)=[P_k(\tau)]^3\times\bn_\tau$, and $\bW_h$, respectively.

\begin{lemma}\label{Lemma5.1} \cite{wy3655} The $L^2$ projections $Q_h$ and ${\mathcal Q}_h$ satisfy the commutative property
 \begin{equation}\label{l}
 \nabla_{w}(Q_h w) = {\mathcalQ}_h(\nabla w), \qquad \forall w\in W^{1,p}(T),
 \end{equation}
 \begin{equation}\label{l-2}
 \nabla_{w}\times(\bbQ_h \bpsi) = {\mathcal Q}_h(\nabla \times \bpsi),  \qquad \forall  \bpsi\in W^{curl, p}(T).
 \end{equation}
 \end{lemma}
\begin{theorem}\label{THM:helmholtz-2} \cite{divcurl} (Helmholtz Decomposition)  
For any vector-valued function $\bu\in [L^p(\Omega)]^3$, there exists a unique $\bpsi\in W_0^{curl,p}(\Omega),\
\phi\in W^{1,p}(\Omega)/\mathbb{R}$, and $\boldeta\in
\mathbb{W}_0^{\varepsilon n,p}(\Omega)$ such that
\begin{eqnarray}\label{EQ:helmholtz-2}
&&\bu =\varepsilon^{-1}\nabla\times\bpsi + \nabla\phi + {\boldsymbol\eta},\\
&&\nabla\cdot(\varepsilon\bpsi)=0,\ \langle
\varepsilon\bpsi\cdot\bn_i, 1\rangle_{\Gamma_i} = 0, \ i=1,\ldots, L.\label{EQ:helmholtz-2.2}
\end{eqnarray}
In addition, there holds  
\begin{equation}\label{EQ:helmholtz-288}
\|\bpsi\|_{W^{\rm curl, p}(\Omega)} + \|\nabla\phi\|_{L^p(\Omega)}\leqC
\|\varepsilon^{\frac{1}{p}} \bu \|_{L^p(\Omega)}.  
\end{equation}
\end{theorem}
\begin{theorem}
The kernel of the matrix of the $L^p$-PDWG method \eqref{EQ:PDWG-3d:01} is given by
$$
K_h = \{ (\bv_h, s_h = 0, \lambda_h = 0, \bq_h = 0):\ \bv_h\in \bV_h\cap \mathbb{W}_0^{\varepsilon n,p}(\Omega)\}.
$$
In other words, the kernel of the matrix of the $L^p$-PDWG scheme \eqref{EQ:PDWG-3d:01} is isomorphic to the subspace of $\mathbb{W}_0^{\varepsilon n,p}(\Omega)$ consisting of harmonic functions that are piecewise polynomial of degree $k$.
\end{theorem}

\begin{proof}  Let $(\bu_h^{(1)}, s_h^{(1)}, \lambda_h^{(1)},\bm{q}_h^{(1)})$ and   $(\bu_h^{(2)}, s_h^{(2)}, \lambda_h^{(2)},\bm{q}_h^{(2)})$ be two different solutions of \eqref{EQ:PDWG-3d:01}. 
This gives, for $i=1,2$, 
\begin{eqnarray}\label{EQ:PDWG-3d:01-1}
 s_1(\lambda_h^{(i)}, \bq_h^{(i)};\varphi,\bpsi) + B_h(\bu_h^{(i)}, s_h^{(i)}; \varphi,\bpsi)&=& F(\varphi,\bpsi),\quad \forall  \varphi\in M_h,\ \bpsi \in \bW_h, \\
-s_2(s_h^{(i)}, r)+B_h(\bm{v}, r; \lambda_h^{(i)}, \bq_h^{(i)}) & =& 0, \quad \forall \bm{v}\in \bV_h,\ r\in S_h. \label{EQ:PDWG-3d:01-2}
\end{eqnarray} 
Given any $j=1,2$,  taking 
 $(\varphi,\bpsi)=(\lambda_h^{(j)},\bq_h^{(j)})$ in \eqref{EQ:PDWG-3d:01-1} and  using \eqref{EQ:PDWG-3d:01-2}, we easily get 
\begin{equation}\label{s1}
\begin{split}
  s_1(\lambda_h^{(i)}, \bq_h^{(i)};\lambda_h^{(j)}, \bq_h^{(j)})+s_2(s_h^{(j)}, s_h^{(i)})
 =  F(\lambda_h^{(j)}, \bq_h^{(j)}),\ \ \forall i, j=1,2.  
\end{split}
\end{equation}
Consequently, for $j=1,2$, 
\begin{equation}\label{s2}
\begin{split}
   s_1(\lambda_h^{(1)}, \bq_h^{(1)};\lambda_h^{(j)}, \bq_h^{(j)})+s_2(s_h^{(j)}, s_h^{(1)})
= s_1(\lambda_h^{(2)}, \bq_h^{(2)};\lambda_h^{(j)}, \bq_h^{(j)}) +s_2(s_h^{(j)}, s_h^{(2)}). 
\end{split}
\end{equation}
Choosing $j=1$ in \eqref{s2} and using  the Young's inequality $|AB|\leq \frac{|A|^p}{p}+\frac{|B|^q}{q}$ yields that 
\begin{equation*}
    \begin{split}
     &    s_1(\lambda_h^{(1)}, \bq_h^{(1)};\lambda_h^{(1)}, \bq_h^{(1)})+s_2(s_h^{(1)}, s_h^{(1)})\leq \frac{ s_1(\lambda_h^{(2)}, \bq_h^{(2)};\lambda_h^{(2)}, \bq_h^{(2)})}{p}\\&+\frac{ s_1(\lambda_h^{(1)}, \bq_h^{(1)};\lambda_h^{(1)}, \bq_h^{(1)})}{q}+\frac{s_2(s_h^{(1)}, s_h^{(1)})}{q}+\frac{s_2(s_h^{(2)}, s_h^{(2)})}{p},
    \end{split}
\end{equation*}
which leads to
  $$s_1(\lambda_h^{(1)}, \bq_h^{(1)};\lambda_h^{(1)}, \bq_h^{(1)})+s_2(s_h^{(1)}, s_h^{(1)})\leq s_1(\lambda_h^{(2)}, \bq_h^{(2)};\lambda_h^{(2)}, \bq_h^{(2)})+ s_2(s_h^{(2)}, s_h^{(2)}).$$
Similarly, we take  $j=2$ in \eqref{s2} and use  the Young's inequality 
again to derive 
  $$  s_1(\lambda_h^{(2)}, \bq_h^{(2)};\lambda_h^{(2)}, \bq_h^{(2)})+ s_2(s_h^{(2)}, s_h^{(2)})\leq s_1(\lambda_h^{(1)}, \bq_h^{(1)};\lambda_h^{(1)}, \bq_h^{(1)})+s_2(s_h^{(1)}, s_h^{(1)}).$$
  Combining the last two inequality leads to 
  \begin{equation}\label{s3} s_1(\lambda_h^{(1)}, \bq_h^{(1)};\lambda_h^{(1)}, \bq_h^{(1)})+s_2(s_h^{(1)}, s_h^{(1)})= s_1(\lambda_h^{(2)}, \bq_h^{(2)};\lambda_h^{(2)}, \bq_h^{(2)})+ s_2(s_h^{(2)}, s_h^{(2)}). 
 \end{equation}
  Note that for any two real numbers $A$ and $B$, there holds 
    $$
   \Big|\frac{A+B}{2}\Big|^p\leq \frac{|A|^p+|B|^p}{2}, 
    $$ and the equality holds true if and only if $A=B$. This follows that
    \begin{equation}\label{s4}
        \begin{split}
           &s_1(\frac{\lambda_h^{(1)}+\lambda_h^{(2)}}{2}, \frac{\bq_h^{(1)}+\bq_h^{(2)}}{2};\frac{\lambda_h^{(1)}+\lambda_h^{(2)}}{2}, \frac{\bq_h^{(1)}+\bq_h^{(2)}}{2})+s_2(\frac{s_h^{(1)}+s_h^{(2)}}{2}, \frac{s_h^{(1)}+s_h^{(2)}}{2})\\
         &  \leq \frac{1}{2}\Big(s_1(\lambda_h^{(1)}, \bq_h^{(1)};\lambda_h^{(1)}, \bq_h^{(1)})+s_1(\lambda_h^{(2)}, \bq_h^{(2)};\lambda_h^{(2)}, \bq_h^{(2)}\Big)+ \frac{1}{2}\Big(s_2(s_h^{(1)}, s_h^{(1)})  + s_2(s_h^{(2)}, s_h^{(2)})\Big).
        \end{split}
    \end{equation}
 On the other hand, a direct calculation from 
     \eqref{s2}-\eqref{s3} yields 
 \begin{equation*}
     \begin{split}
      &s_1(\lambda_h^{(1)}, \bq_h^{(1)};\lambda_h^{(1)}, \bq_h^{(1)})+s_2(s_h^{(1)}, s_h^{(1)})
     \\
     =&\frac{1}{2}\Big( s_1(\lambda_h^{(1)}, \bq_h^{(1)};\lambda_h^{(1)}, \bq_h^{(1)})+s_2(s_h^{(1)}, s_h^{(1)})+ s_1(\lambda_h^{(2)}, \bq_h^{(2)};\lambda_h^{(1)}, \bq_h^{(1)}) +s_2(s_h^{(1)}, s_h^{(2)}) \Big)\\
     =& s_1(\frac{\lambda_h^{(1)}+\lambda_h^{(2)}}{2} , \frac{ \bq_h^{(1)}+ \bq_h^{(2)}}{2};\lambda_h^{(1)}, \bq_h^{(1)})+s_2(s_h^{(1)}, \frac{s_h^{(1)}+s_h^{(2)}}{2}).
     \end{split}
 \end{equation*}
   Using the  Young's inequality again,  we get 
 \begin{equation*}
     \begin{split}
     &s_1(\lambda_h^{(1)}, \bq_h^{(1)};\lambda_h^{(1)}, \bq_h^{(1)})+s_2(s_h^{(1)}, s_h^{(1)}) \\
     \leq&  s_1(\frac{\lambda_h^{(1)}+\lambda_h^{(2)}}{2}, \frac{ \bq_h^{(1)}+ \bq_h^{(2)}}{2};\frac{\lambda_h^{(1)}+\lambda_h^{(2)}}{2}, \frac{ \bq_h^{(1)}+ \bq_h^{(2)}}{2})+s_2(\frac{s_h^{(1)}+s_h^{(2)}}{2}, \frac{s_h^{(1)}+s_h^{(2)}}{2}).
        \end{split}
 \end{equation*}
Then we conclude from \eqref{s4} and \eqref{s3} 
 \begin{equation*}\label{ss1}
     \begin{split}
     &s_1(\frac{\lambda_h^{(1)}+\lambda_h^{(2)}}{2}, \frac{ \bq_h^{(1)}+ \bq_h^{(2)}}{2};\frac{\lambda_h^{(1)}+\lambda_h^{(2)}}{2}, \frac{ \bq_h^{(1)}+ \bq_h^{(2)}}{2})+s_2(\frac{s_h^{(1)}+s_h^{(2)}}{2}, \frac{s_h^{(1)}+s_h^{(2)}}{2})
     \\=& s_1(\lambda_h^{(1)}, \bq_h^{(1)};\lambda_h^{(1)}, \bq_h^{(1)})+s_2(s_h^{(1)}, s_h^{(1)})
     =  s_1(\lambda_h^{(2)}, \bq_h^{(2)};\lambda_h^{(2)}, \bq_h^{(2)})+ s_2(s_h^{(2)}, s_h^{(2)}).
         \end{split}
 \end{equation*} 
The above equation holds true if and only if
\begin{eqnarray}\label{eq1}
  \lambda_0^{(1)}- \lambda_b^{(1)}&=& \lambda_0^{(2)}- \lambda_b^{(2)}, \qquad \text{on}\ \partial T,
 \\
 \bq_0^{(1)}\times \bn- \bq_b^{(1)}\times \bn&=&\bq_0^{(2)}\times \bn- \bq_b^{(2)}\times \bn, \qquad \text{on}\ \partial T, 
 \\
 s_0^{(1)}-s_b^{(1)}&=&s_0^{(2)}-s_b^{(2)}, \qquad \text{on}\ \partial T,\label{eq3}
  \end{eqnarray}
  
  Denoting $\epsilon_h=\lambda_h^{(1)}-\lambda_h^{(2)}=\{\epsilon_0, \epsilon_b\}$,  $\bE_h =\bq_h^{(1)}-\bq_h^{(2)}=\{\bE_0, \bE_b\}$, $e_h =s_h^{(1)}-s_h^{(2)}=\{e_0, e_b\}$, we have 
 \begin{eqnarray}\label{e1}
 \epsilon_0  =  \epsilon_b, \  \text{on}\ \partial T, \qquad
 \bE_0 \times \bn  =  \bE_b\times \bn, \  \text{on}\ \partial T,  \qquad
  e_0  = e_b, \ \text{on}\ \partial T.  
 \end{eqnarray}
   Since $s_0^{(1)}-s_b^{(1)}=s_0^{(2)}-s_b^{(2)}$ on $\partial T$,  there holds 
 $$
 s_2(s_h^{(1)}, r)=s_2(s_h^{(2)}, r), \qquad \forall r\in S_h,
 $$
 which, combined with \eqref{EQ:PDWG-3d:01-2},  gives
$$
B_h(\bv, r; \lambda_h^{(1)}, \bq_h^{(1)})= B_h(\bv, r; \lambda_h^{(2)}, \bq_h^{(2)}), \qquad \forall \bv\in \bV_h, r\in S_h,
$$
or equivalently, 
 $$
B_h(\bv, r; \epsilon_h, \bE_h)= 0,  \qquad\forall \bv\in \bV_h, r\in S_h,
$$
 i.e., 
\begin{eqnarray}
(\bE_0, \varepsilon\nabla_w r)+(\bm{v}, \varepsilon \nabla_w \epsilon_h +\nabla_w \times \bE_h) = 0,\qquad \forall \bm{v}\in\bV_h, r\in S_h.\label{eq:11:08:102}
\end{eqnarray}
It follows from \eqref{e1} that
 $\epsilon_0\in C(\Omega)$, $e_0\in C(\Omega)$ and $\bE_0\in H_0(curl;\Omega)$, which indicates
\begin{eqnarray}\label{EQ:Nov-26:01}
\nabla \epsilon_0 = \nabla_w \epsilon_h ,\; \nabla \times \bE_0 = \nabla_w \times \bE_h.
\end{eqnarray}
Letting $r=0$ and varying $\bm{v}$ in \eqref{eq:11:08:102}, we get
$$
\varepsilon \nabla_w \epsilon_h + \nabla_w \times \bE_h = 0,
$$
which, together with \eqref{EQ:Nov-26:01}, gives
\begin{equation}\label{EQ:11:10:100}
\varepsilon \nabla \epsilon_0 + \nabla\times \bE_0 = 0.
\end{equation}
Using $\bE_0\in H_{0}(curl;\Omega)$, we have
\begin{eqnarray*}
(\varepsilon \nabla \epsilon_0 + \nabla\times \bE_0, \nabla \epsilon_0 )
&=&(\varepsilon \nabla \epsilon_0, \nabla \epsilon_0) + (\nabla\times \bE_0, \nabla \epsilon_0)\\
&=&(\varepsilon \nabla \epsilon_0, \nabla \epsilon_0) + \langle \bm{n}\times \bE_0, \epsilon_0\rangle
= (\varepsilon \nabla \epsilon_0, \nabla \epsilon_0),\end{eqnarray*}
which, from \eqref{EQ:11:10:100}, implies
 $\nabla \epsilon_0 = \bm{0}$, and hence $\epsilon_0 \equiv 0$ as a function with mean value 0.  This further leads to $\epsilon_b\equiv 0$. Thus, from \eqref{EQ:11:10:100} we have
$$
\nabla \times \bE_0=0, \qquad \text{in}\ \Omega.
$$
Note that $\bE_0$ satisfies
$$
(\bE_0, \varepsilon\nabla_w r)=0,\qquad \forall r\in S_h.
$$
 This leads to $\bE_0 \in H(div_\varepsilon; \Omega)$ and
$$
\nabla\cdot(\varepsilon\bE_0 )=0,\quad \langle \bE_0\cdot\bn_i,1\rangle_{\Gamma_i}=0, i=1,2,\cdots L.
$$
This, together with $\nabla\times\bE_0=0$ and $\bE_0 \in H_0(curl;\Omega)$, indicates that $\bE_0 \equiv 0$, and further $\bE_{b}=\bn\times(\bE_{b}\times\bn)=\bn\times 0=0$.

  Using \eqref{eq1}-\eqref{eq3}, we have
 $$
 s_1(\lambda_h^{(1)}, \bq_h^{(1)}; \varphi, \bpsi)= s_1(\lambda_h^{(2)}, \bq_h^{(2)}; \varphi, \bpsi),
 $$
 which yields, together with \eqref{EQ:PDWG-3d:01-1},  
 $$B_h(\bu_h^{(1)}, s_h^{(1)}; \varphi,\bpsi) =B_h(\bu_h^{(2)}, s_h^{(2)}; \varphi,\bpsi) ,\quad \forall  \varphi\in M_h,\ \bpsi \in \bW_h.$$
Denote $\bE_{\bu_h}=\bu_h^{(1)}-\bu_h^{(2)}$. The above equality is equivalent to
 \begin{equation}\label{ss}
   0=  B_h(\bE_{\bu_h}, e_h; \varphi,\bpsi)=(\bE_{\bu_h}, \varepsilon\nabla_w\varphi+\nabla_w\times\bpsi)+(\bpsi_0, \varepsilon \nabla_w e_h),\quad \forall  \varphi\in M_h,\ \bpsi \in \bW_h.
 \end{equation}  
 Now we have, from the Helmholtz decomposition \eqref{EQ:helmholtz-2}, 
$$
\bE_{\bu_h}=\varepsilon^{-1} \nabla\times\bm{\tilde\psi}+\nabla\tilde\phi+\bm{\tilde\eta},
$$
where $\bm{\tilde\eta}\in \mathbb{W}_0^{\varepsilon n,p}(\Omega)$ and $\bm{\tilde\psi}\in W_0^{curl,p}(\Omega)$ satisfying $\nabla\cdot(\varepsilon\bm{\tilde\psi}) =0$ and $\langle\varepsilon\bm{\tilde\psi}\cdot\bn_i,1\rangle_{\Gamma_i}=0$ for $i=1,\cdots, L$.
It follows from  $e_0=e_b$ on $\pT$ for each element $T\in \T_h$   that $e_0\in W^{1, p}(\Omega)$. This leads to $\nabla_w e_h = \nabla e_0$. If the dimension of $\mathbb{W}_0^{\varepsilon n,p}(\Omega)$ is  0,  we have $\bm{\tilde\eta}=0$.   Letting the test functions $\phi$ and $\bm{\psi}$ in \eqref{ss}  be the $L^2$ projections of the corresponding function in the Helmholtz decomposition gives rise to 
\begin{equation}
\begin{split}
0=&(\bE_{\bu_h}, \varepsilon\nabla_w Q_h\tilde\varphi + \nabla_w \times \bbQ_h\bm{\tilde\psi})+(\bbQ_0 \bm{\tilde\psi},\varepsilon\nabla_w e_h)\\
=&(\bE_{\bu_h}, {\mathcal Q}_h\varepsilon\nabla\tilde\varphi + {\mathcal Q}_h\nabla \times \bm{\tilde\psi}) +(\bm{\tilde\psi},\varepsilon\nabla e_0)\\
=&(\bE_{\bu_h}, \varepsilon\nabla\tilde\varphi + \nabla \times \bm{\tilde\psi}) +(\bm{\tilde\psi},\varepsilon\nabla e_0)\\
=&(\varepsilon\bE_{\bu_h}, \bE_{\bu_h} - \bm{\tilde\eta}) +(\bm{\tilde\psi},\varepsilon\nabla e_0)\\
=&(\varepsilon(\bE_{\bu_h}-\bm{\tilde\eta}), \bE_{\bu_h}- \bm{\tilde\eta}),
\end{split}
\end{equation}
which leads to $\bE_{\bu_h}-\bm{\tilde\eta}=0$,  i.e., $\bE_{\bu_h}$ is a harmonic function. As a harmonic function in the form of piecewise polynomial of degree $k$, the first term on the right-hand side of \eqref{ss} is zero for any test functions $\varphi\in M_h$ and $\bm{\psi}\in\bW_h$, which further implies that $\nabla_w e_h=0$. Using \eqref{e1} gives $\nabla e_0= \nabla_w e_h =0$. Therefore we obtain $e_0\equiv 0$ and further $e_b \equiv 0$.

This completes the proof of the theorem.
\end{proof}

Our main result for the solution existence and uniqueness of the numerical scheme \eqref{EQ:PDWG-3d:01} is stated as follows.

\begin{theorem}
The $L^p$-PDWG finite element scheme \eqref{EQ:PDWG-3d:01} has a unique solution for $s_h$, $\lambda_h$ and $\bq_h$. The solution $\bm{u}_h$ is unique up to a harmonic function $\bm{\eta}_h\in \mathbb{W}_0^{\varepsilon n,p}(\Omega)$ which is a piecewise polynomial of degree $k$.
\end{theorem}

\begin{remark} For the   lowest order  $k=0$ of the $L^p$-PDWG scheme \eqref{EQ:PDWG-3d:01},   any $\bm{\eta}_h$ in the kernel $K_h$ of the matrix of the $L^p$-PDWG method  is a piecewise constant vector field.   $\bm{\eta}_h$ is thus continuous across each interior element interface and has vanishing value on the domain boundary along the normal direction. This leads to $\bm{\eta}_h\equiv 0$. Therefore, the $L^p$-PDWG finite element scheme \eqref{EQ:PDWG-3d:01} has a unique solution for  $\bm{u}_h$ in the case of the lowest order element. 
\end{remark}


\section{$L^q$-Error Analysis for the Primal Variable}\label{Section:7}
In this section, we shall establish the $L^q$ error estimates for primal variable $\bu_h$ in the $L^p$-PDWG scheme \eqref{EQ:PDWG-3d:01}.
Denote the error functions by
$$
e_{\bu}= \mathcalQ_h\bu - \bu_h, \ e_s=Q_h s - s_h,\ e_\lambda=Q_h\lambda - \lambda_h, \ e_{\bm{q}}=\bbQ_h\bm{q}-\bm{q}_h.
$$
We begin with the study of error equations  for the $L^p$-PDWG scheme \eqref{EQ:PDWG-3d:01} developed for the  div-curl system \eqref{EQ:div-curl:div-eq}-\eqref{EQ:div-curl:normalBC}.
\subsection{Error Equations}
For the exact solution $\{\bu; s=0\}$ of the div-curl system, recalling the definition of $B_h(\cdot;\cdot)$ in \eqref{EQ:Bh} and using 
 (\ref{EQ:dis_WeakGradient}) and (\ref{EQ:dis_WeakCurl}), we have 
\begin{eqnarray*}
B_h(\mathcalQ_h\bu, Q_h s; \varphi,\bpsi) 
&=& (\mathcalQ_h\bu, \varepsilon \nabla \varphi_0 + \nabla \times \bpsi_0)  
  + \langle \mathcalQ_h\bu, \varepsilon\bn(\varphi_b-\varphi_0) + (\bpsi_0-\bpsi_b)\times\bn\rangle_{{\mathcal E}_h}\\
&=& (\bu, \varepsilon \nabla \varphi_0 + \nabla \times \bpsi_0)  
  + \langle \mathcalQ_h\bu, \varepsilon\bn(\varphi_b-\varphi_0) + (\bpsi_0-\bpsi_b)\times\bn\rangle_{{\mathcal E}_h}. 
\end{eqnarray*}
By using the integration by parts and \eqref{EQ:div-curl:div-eq}-\eqref{EQ:div-curl:normalBC}, we easily obtain 
\begin{eqnarray*} 
&&B_h(\mathcalQ_h\bu, Q_h s; \varphi,\bpsi) \\
&&= -(\nabla\cdot(\varepsilon\bu),\varphi_0) + (\nabla\times\bu, \bpsi_0)  
  + \langle \bu, \varepsilon\bn(\varphi_0-\varphi_b) + (\bpsi_b-\bpsi_0)\times\bn\rangle_{{\mathcal E}_h}\\
&& + \langle \mathcalQ_h\bu, \varepsilon\bn(\varphi_b-\varphi_0) + (\bpsi_0-\bpsi_b)\times\bn\rangle_{{\mathcal E}_h}  + \langle {\phi}_1, \varphi_b\rangle_{\partial\Omega}\\
&&= \langle {\phi}_1, {\varphi}_b\rangle_{\partial\Omega} - (f,\varphi_0) +(\bm{g},\bm{\psi}_0) 
 + \langle \bu-\mathcalQ_h\bu, \varepsilon\bn(\varphi_0-\varphi_b) + (\bpsi_b-\bpsi_0)\times\bn\rangle_{{\mathcal E}_h}.
\end{eqnarray*}
Noticing that $\lambda=0$ and $\bm{q}=0$,   then 
\begin{equation}\label{EQ:div-curl:EE:November-08:ee:01}
\begin{split}
 s_1(e_{\lambda}, e_{\bm{q}};\varphi, \bm{\psi})
+ B_h(e_{\bu} , e_s; \varphi,\bpsi)  
=   \langle \bu-\mathcalQ_h\bu, \varepsilon\bn(\varphi_0-\varphi_b) + (\bpsi_b-\bpsi_0)\times\bn\rangle_{{\mathcal E}_h}.
\end{split}
\end{equation}
Similarly, we conclude from the fact $s=0, \ \bm{q}=0$, $\lambda=0$ that 
\begin{equation}\label{EQ:div-curl:EE:November-08:ee:02}
-s_2(e_s, r)
+ B_h(\bm{v}, r; e_{\lambda}, e_ {\bm{q}}) = 0.
\end{equation}
 
The above two equations \eqref{EQ:div-curl:EE:November-08:ee:01}-\eqref{EQ:div-curl:EE:November-08:ee:02} are the error equations for the $L^p$-PDWG scheme \eqref{EQ:PDWG-3d:01}, which will be frequently used in our error estimates.

\subsection{Error estimates for the dual variables}

Recall that $\T_h$ is a shape-regular finite element partition of
the domain $\Omega$. For any $T\in\T_h$ and $\nabla w\in L^{q}(T)$ with $q>1$, the following trace inequality holds true:
\begin{equation}\label{trace-inequality}
\|w\|^q_{L^q(\pT)}\leq C
h_T^{-1}\Big(\|w\|_{L^{q}(T)}^q+h_T^{q} \| \nabla w\|_{L^{q}(T)}^q\Big).
\end{equation}
 By using the Cauchy-Schwarz inequality and the trace inequality, we get 
 \begin{eqnarray}\label{new:1}
 \begin{split}
|\langle w,v\rangle_{{\mathcal E}_h} |
 &\leq   (\sum_{T\in \mathcal {T}_h } \| w\|^q_{L^q(\partial T)} )^{\frac{1}{q}}
(\sum_{T\in \mathcal {T}_h } \|v \|^p_{L^p(\partial T)} )^{\frac{1}{p}} \\
 &\leq   C
h^{-\frac 1q}(\|w\|_{L^{q}(T)}+h \| \nabla w\|_{L^{q}(T)})(\sum_{T\in \mathcal {T}_h } \|v \|^p_{L^p(\partial T)} )^{\frac{1}{p}}. 
\end{split}
\end{eqnarray}
  
  Now we are ready to present the error estimates for the dual variables. 

\begin{theorem}\label{THM:ErrorEstimate4Triple} Assume  the solution of the div-curl system (\ref{EQ:div-curl:div-eq})-(\ref{EQ:div-curl:normalBC}) satisfies  $\bu\in [W^{k+\theta, q}(\Omega)]^3$ for $\theta\in (1/2,1]$.
 For the numerical solution $\bu_h,s_h,\lambda_h,\bm{q}_h$ arising from the $L^p$-PDWG scheme \eqref{EQ:PDWG-3d:01}, there holds  
\begin{eqnarray}\label{er1}
\3bar (e_\lambda, e_{\bm{q}})\3bar &\leq&   Ch^{(k+\theta)\frac{q}{  p} }\|\nabla^{k+\theta}\bu\|^{\frac{q}{p}}_{L^q(\Omega)}, \\
 \3bar e_s \3bar  & \leq & Ch^{(k+\theta)  }\|\nabla^{k+\theta}\bu\| _{L^q(\Omega)}. \label{er2} 
\end{eqnarray}
 \end{theorem}
\begin{proof}First, we have, from the error equations
\eqref{EQ:div-curl:EE:November-08:ee:01}-\eqref{EQ:div-curl:EE:November-08:ee:02} that 
\begin{equation}\label{r1}
\begin{split}
s_1(e_\lambda, e_{\bm{q}}; e_\lambda, e_{\bm{q}}) + s_2(e_s, e_s)
= &\langle \bu-\mathcalQ_h\bu, \varepsilon\bn(e_{\lambda,0}-e_{\lambda,b}) + (e_{\bm{q},b}-e_{\bm{q},0})\times\bn\rangle_{{\mathcal E}_h}.
\end{split}
\end{equation}
  In light of \eqref{new:1} and using the approximation property of $\mathcalQ_h$, we get 
\begin{equation}\label{te3}
\begin{split}
     &|\langle \bu-\mathcalQ_h\bu, \varepsilon\bn(e_{\lambda,0}-e_{\lambda,b}) + (e_{\bm{q},b}-e_{\bm{q},0})\times\bn\rangle_{{\mathcal E}_h}| \\
 \leq & Ch^{k+\theta-\frac{1}{q}}\|\nabla^{k+\theta}\bu\|_{L^q} \Big(\sum_{T\in \mathcal {T}_h } (\|\varepsilon\bn(e_{\lambda,0}-e_{\lambda,b}) \|^p_{L^p(\partial T)}+\|(e_{\bm{q},b}-e_{\bm{q},0})\times\bn \|^p_{L^p(\partial T)} )\Big)^{\frac{1}{p}}\\
 \leq & Ch^{k+\theta}\|\nabla^{k+\theta}\bu\|_{L^q} \3bar ( e_\lambda, e_{\bm{q}}) \3bar.
  \end{split}
\end{equation}
Substituting the above inequality into \eqref{r1} and using the Cauchy-Schwarz inequality yields 
\begin{equation}\label{EQ:error-estimate-n-part01}
\3bar (e_\lambda, e_{\bm{q}})\3bar^p + \3bar e_s \3bar^q \leq C_1h^{(k+\theta)q }\|\nabla^{k+\theta}\bu\|^q_{L^q(\Omega)}.
\end{equation}
  Then \eqref{er1}-\eqref{er2} follow directly. 
\end{proof}

\subsection{$L^q$ error estimates for the primal variable ${\bf u}_h$.}

To derive the $L^q$-estimate for the error function $e_\bu$,  we need employ  the Helmholtz decomposition \eqref{EQ:helmholtz-2} for any function $\bv$, such that  
\begin{equation}\label{EQ:2020-02-09:101}
 \bv = \varepsilon^{-1} \nabla\times\bm{\tilde\psi} + \nabla\tilde\phi + \tilde\boldeta,
\end{equation}  
where $\tilde\phi\in W^{1, p}(\Omega)$, $\bm{\tilde\psi}\in {W}_0^{curl,p}(\Omega)$, and $\tilde\boldeta\in \mathbb{W}_0^{\varepsilon n,p}(\Omega)$. 
We assume the  $W^{\alpha,p}$-regularity holds true for some fixed $\alpha\in (1/2, 1]$:
\begin{equation}\label{EQ:regularity-assumption-helmholtz-01new}
\|\bm{\tilde\psi}\|_{\alpha,p} + \|\tilde\phi\|_{\alpha, p} \leq C\|\bv - \tilde\boldeta\|_{0,p}.
\end{equation}
The main convergence result of this paper is stated as follows.

\begin{theorem}\label{THM:ErrorEstimate4uh} Let $\bu$ be a solution of the div-curl system (\ref{EQ:div-curl:div-eq})-(\ref{EQ:div-curl:normalBC}) such that $\bu\in [W^{k+\theta, q}(\Omega)]^3$ for $\theta\in (1/2,1]$. Assume that the Helmholtz decomposition \eqref{EQ:2020-02-09:101} has the $W^{\alpha,p}$-regularity estimate \eqref{EQ:regularity-assumption-helmholtz-01new}. For a numerical solution $\bu_h, \ s_h, \ \lambda_h,\ \bm{q}_h$ arising from $L^p$-PDWG scheme \eqref{EQ:PDWG-3d:01}, there exists a harmonic function $ \tilde\boldeta \in \mathbb{W}_0^{\varepsilon n,p}(\Omega)$ such that 
\begin{equation}\label{EQ:2020-02-09:102}
\begin{split}
\|\varepsilon^{\frac{1}{q}} (\bu_h+\tilde\boldeta- \mathcalQ_h\bu )\|_{L^q(\Omega)} \leq Ch^{k+\theta+\alpha-1}\|\nabla ^{k+\theta}\bu\|_{L^q(\Omega)}.
 \end{split}
\end{equation} 
\end{theorem}

\begin{proof}
Given any function $ \bv$,  let $\tilde\phi\in W^{1, p}(\Omega)$, $\bm{\tilde\psi}\in {W}_0^{curl,p}(\Omega)$, and $\tilde\boldeta\in \mathbb{W}_0^{\varepsilon n,p}(\Omega)$
satisfy \eqref{EQ:2020-02-09:101}. 
Taking $\varphi=Q_h\tilde\phi$ and $\bm{\psi}=Q_h\bm{\tilde\psi}$ in $B_h(e_\bu, e_s; \varphi,\bm{\psi})$ and using Lemma \ref{Lemma5.1} gives 
\begin{equation*}\label{EQ:11:08:300}
\begin{split}
B_h(e_\bu, e_s; \varphi,\bm{\psi})=&(e_\bu, \varepsilon\nabla_w Q_h\tilde\phi + \nabla_w\times \bbQ_h\bm{\tilde\psi})+(\bbQ_0\bm{\tilde\psi},\varepsilon\nabla_w e_s)\\
=&(e_\bu, \varepsilon\mathcalQ_h\nabla_w \tilde\phi + \mathcalQ_h\nabla_w\times \bm{\tilde\psi})+(\bbQ_0\bm{\tilde\psi},\varepsilon\nabla_w e_s)\\
= & (e_\bu, \varepsilon\nabla \tilde\phi + \nabla\times \bm{\tilde\psi})+(\bbQ_0\bm{\tilde\psi},\varepsilon\nabla_w e_s).
\end{split}
\end{equation*}
 By using the Helmholtz decomposition \eqref{EQ:2020-02-09:101}, we have
 \begin{eqnarray}
 \begin{split}
   B_h(e_\bu, e_s; \varphi,\bm{\psi})&=  (\varepsilon e_\bu,  \bv - \tilde\boldeta ) + (\varepsilon \bbQ_0\bm{\tilde\psi},\nabla_w e_s)\\
&=  (\varepsilon (e_\bu - \tilde\boldeta), \bv - \tilde\boldeta ) + (\varepsilon \bbQ_0\bm{\tilde\psi},\nabla_w e_s).
\end{split}
 \end{eqnarray}
 
  From the definition of the weak gradient, we have
\begin{equation*}
\begin{split}
(\varepsilon \bbQ_0\bm{\tilde\psi},\nabla_w e_s) = &(\varepsilon \bbQ_0\bm{\tilde\psi},\nabla e_{s,0}) + \langle \varepsilon \bbQ_0\bm{\tilde\psi}\cdot\bn, e_{s,b}-e_{s,0}\rangle_{{\mathcal E}_h} \\
= & (\varepsilon\bm{\tilde\psi},\nabla e_{s,0}) + \langle\varepsilon \bbQ_0\bm{\tilde\psi}\cdot\bn, e_{s,b}-e_{s,0}\rangle_{{\mathcal E}_h} \\
= &  -(\nabla\cdot(\varepsilon\bm{\tilde\psi}),  e_{s,0}) +\langle \varepsilon\bm{\tilde\psi}\cdot\bn, e_{s,0} \rangle_{{\mathcal E}_h} + \langle \varepsilon \bbQ_0\bm{\tilde\psi}\cdot\bn, e_{s,b}-e_{s,0}\rangle_{{\mathcal E}_h}\\
= & \langle \varepsilon\bm{\tilde\psi}\cdot\bn, e_{s,0} - e_{s,b}\rangle_{{\mathcal E}_h} + \langle\varepsilon \bbQ_0\bm{\tilde\psi}\cdot\bn, e_{s,b}-e_{s,0}\rangle_{{\mathcal E}_h}\\
= & \langle \varepsilon(\bm{\tilde\psi} - \bbQ_0\bm{\tilde\psi})\cdot\bn, e_{s,0} - e_{s,b}\rangle_{{\mathcal E}_h}.
\end{split}
\end{equation*}
Substituting the above into \eqref{EQ:11:08:300}   yields
\begin{equation}\label{EQ:11:05:300}
\begin{split}
(\varepsilon (e_\bu - \tilde\boldeta), \bv - \tilde\boldeta ) = & B(e_\bu, e_s; \varphi,\bm{\psi}) - \langle \varepsilon(\bm{\tilde\psi} - \bbQ_0\bm{\tilde\psi})\cdot\bn, e_{s,0} - e_{s,b}\rangle_{{\mathcal E}_h}\\
= & \langle \bu-\mathcalQ_h\bu, \varepsilon\bn(\varphi_0-\varphi_b) + (\bpsi_b-\bpsi_0)\times\bn\rangle_{{\mathcal E}_h} - s_1(e_\lambda, e_{\bm{q}};\varphi, \bm{\psi}) \\
& - \langle\varepsilon (\bm{\tilde\psi} - \bbQ_0\bm{\tilde\psi})\cdot\bn, e_{s,0} - e_{s,b}\rangle_{{\mathcal E}_h}
\\=&I_1+I_2+I_3,
\end{split}
\end{equation}
where we used the first error equation \eqref{EQ:div-curl:EE:November-08:ee:01}, $I_i (i=1,\cdots, 3)$ are defined accordingly.

As to $I_1$, using the same argument as what we did for \eqref{te3}, we  get 
\begin{equation}\label{i1}
     |I_1|    \leq Ch^{k+\theta}\|\nabla ^{k+\theta} \bu\|_{L^q(\Omega)}\3bar (\varphi, \bpsi)\3bar.
\end{equation}
As to $I_2$,  recalling the definition of $s_1$ and using the Cauchy-Schwarz inequality,  we have
\begin{equation}\label{i2}
\begin{split}
    | I_2| \leq & C \Big(\sum_{T\in \T_h} h_T^{1-p} \int_{\partial T} |e_{\lambda, 0} -e_{\lambda,b}|^{q(p-1)}ds \Big)^{\frac{1}{q}} \Big(\sum_{T\in \T_h}  \int_{\partial T}h_T^{1-p}  |\varphi_0 - \varphi_b|^{p}ds \Big)^{\frac{1}{p}}\\
                                  &+C  \Big(\sum_{T\in \T_h} h_T^{1-p}  \int_{\partial T} |e_{\bm{q},0} \times \bm{n} -e_{\bm{q},b} \times \bm{n}|^{q(p-1)}ds\Big)^{\frac{1}{q}}\\\cdot &\Big(\sum_{T\in \T_h}  \int_{\partial T}h_T^{1-p} |\bm{\psi}_0 \times \bm{n} -\bm{\psi}_{b} \times \bm{n}|^p ds\Big)^{\frac{1}{p}}\\
                                  \leq &C\3bar (e_\lambda, e_{\bm{q}})\3bar^{\frac{p}{q}} \3bar(\varphi, \bm{\psi})\3bar.
\end{split} 
\end{equation}
Similarly,    we use  \eqref{new:1} and the approximation property of $\bbQ_0$ to get that  
\begin{equation}\label{i3}
    \begin{split}
      | I_3| 
     \leq& Ch^{\alpha-\frac{1}{p}}\|\nabla ^{\alpha} \bm{\tilde\psi}\|_{L^p(\Omega)}h^{\frac{q-1}{q}} \3bar e_s \3bar
= Ch^{\alpha}\|\nabla ^{\alpha} \bm{\tilde\psi}\|_{L^p(\Omega)} \3bar e_s \3bar.
     \end{split} 
\end{equation}
It is easy to check
\begin{equation}\label{i4}
\3bar (\varphi, \bpsi)\3bar \leq Ch^{\alpha-1}( \|\tilde\phi\|_{\alpha, p}+\| \tilde\bpsi\|_{\alpha, p}).
\end{equation}
Substituting \eqref{i1}-\eqref{i3} into \eqref{EQ:11:05:300}, and using \eqref{er1}, \eqref{er2}, \eqref{EQ:regularity-assumption-helmholtz-01new} and \eqref{i4}, this gives
\begin{equation*}
\begin{split}
|(\varepsilon (e_\bu - \tilde\boldeta), \bv - \tilde\boldeta )|
\leq
& Ch^{k+\theta}\|\nabla ^{k+\theta} \bu\|_{L^q(\Omega)}\3bar (\varphi, \bpsi)\3bar+Ch^{\alpha}\|\nabla ^{\alpha} \bm{\tilde\psi}\|_{L^p(\Omega)} \3bar e_s \3bar\\ 
\leq & Ch^{k+\theta+\alpha-1}\|\nabla ^{k+\theta}\bu\|_{L^q(\Omega)} \|\bv - \tilde\boldeta\|_{0,p}.
\end{split}
\end{equation*}
It follows that
\begin{equation*}
\begin{split}
\|\varepsilon^{\frac{1}{q}} (e_\bu-\tilde\boldeta)\|_{L^q(\Omega)} \leq Ch^{k+\theta+\alpha-1}\|\nabla ^{k+\theta}\bu\|_{L^q(\Omega)},
\end{split}
\end{equation*}
which gives rise to the error estimate \eqref{EQ:2020-02-09:102}.
This completes the proof of the theorem.
\end{proof}

\section{Numerical Experiments} \label{Section:8}
 In this section, we present some numerical examples to test the performance and accuracy of the PDWG method proposed in \eqref{EQ:PDWG-3d:01}. 
 In our numerical experiments,  the
 computation domain is first partitioned into cubes, and then each cube is divided into $6$ tetrahedra of equi-volume.   
We choose the discontinuous piecewise constant vector fields to approximate the exact solution ${\bf u}$.  That is,  the finite element spaces are given as follows: 
  \begin{equation*}
\begin{split}
\bV_h=&\{\bm{v}:\ \bm{v}|_T\in [P_0(T)]^3, \forall T\in\T_h \},\\
S_h=&\{ s: \ s|_{T}=\{s_0,s_b\}\in \{P_0(T), \Pi_{i=1}^4P_0(F_i)\}, \forall T\in\T_h, F_i\in\partial T \},\\
M_h=&\{ \varphi: \varphi|_{T}=\{\varphi_0,\varphi_b\}\in \{P_0(T), \Pi_{i=1}^4P_0(F_i)\}, \forall T\in\T_h, F_i\in\partial T \},\\
\bW_h=&\{\bm{\psi}: \bm{\psi}|_T=\{\bm{\psi}_0,\bm{\psi}_{b} \}\in\{  [P_0(T)]^3, T_0(\pT)\}, \forall T\in\T_h\},
\end{split}
\end{equation*}
where $T_0(\pT)$ is the the tangent space of $\pT$ given by 
\[
   T_0(\pT)=\{\bm{\psi}: \bm{\psi_{i,j}}\in [P_0(F_i)]^3\times {\bf n}_{F_i},\ \ F_i\in\partial T, i=1,2,3,4, j=1,2\}. 
\]
Here ${\bf n}_{F_i}$ denotes the outer unit normal vector to face $F_i$. 

To solve the system of nonlinear equation \eqref{EQ:PDWG-3d:01},  we adopt 
 an iterative scheme similar to that for the $L^1$ minimization problem in \cite{Vogel}. Specifically, given an approximation 
 $(\bu^m_h, \lambda^m_h,s_h^m, \bq_h^m)\in \bV_h\times M_h\times S_h\times  \bW_h$  at step $m$, the scheme shall compute a new approximate solution  $(\bu^{m+1}_h, \lambda^{m+1}_h,s_h^{m+1}, \bq_h^{m+1})\in \bV_h\times M_h\times S_h\times  \bW_h$ such that
\begin{equation}\label{iter:1}
\left\{
\begin{array}{rl}
s_1(\lambda_h, \bq_h;\varphi,\bpsi) + B_h(\bu_h, s_h; \varphi,\bpsi)&= F(\varphi,\bpsi),\quad \forall  \varphi\in M_h,\ \bpsi \in \bW_h, \\
-s_2(s_h, r)+B_h(\bm{v}, r; \lambda_h, \bq_h) & = 0, \qquad\qquad \forall \bm{v}\in \bV_h,\ r\in S_h.
\end{array}
\right.
\end{equation}
 where
\begin{equation*}\label{stab3}
\begin{split}
&s_1(\lambda_h, \bq_h;\varphi, \bm{\psi}) = \rho_1 \sum_{T\in \T_h} h_T^{1-p}  \int_{\partial T} (|\lambda^m_0-\lambda^m_b|+\epsilon_0)^{p-2}(\lambda^{m+1}_0-\lambda^{m+1}_b)(\varphi_0 - \varphi_b)ds  \\
                                  &+ \rho_2 \sum_{T\in \T_h}  h_T^{1-p}  \int_{\partial T}(|\bm{q}^m_0 \times \bm{n} -\bm{q}^m_{b} \times \bm{n}|+\epsilon_0)^{p-2}(\bm{q}^{m+1}_0 \times \bm{n} -\bm{q}^{m+1}_{b} \times \bm{n})(\bm{\psi}_0 \times \bm{n} -\bm{\psi}_{b} \times \bm{n})ds, \\
    &s_2(s_h; r)=\rho_3 \sum_{T\in \T_h}h_T^{1-q}\int_{\partial T} (|s^m_0 -s^m_b|+\epsilon_0)^{q-2}(s^{m+1}_0 -s^{m+1}_b)( r_0 - r_b)ds.                              
\end{split}
\end{equation*}
Here $\epsilon_0$ is a small, but positive constant, and $ \rho_i, i\le 3$ are some positive stabilization parameters. 

We would like to point out that although $ \rho_i, i\le 3$  in the algorithm  \eqref{EQ:PDWG-3d:01} could be arbitrary,
 the iterative scheme \eqref{iter:1} is not convergent for any positive  $ \rho_i, i\le 3$.  Our numerical experiments indicate that 
 $ \rho_i, i=1,2$ should be taken large enough to ensure the convergence of the iterative scheme. 

 In our experiments,  we test  various problems in which the exact solution ${\bf u}$ has different 
regularities and the computational domain includes convex, non-convex polyhedral regions and cavities.  We shall evaluate the errors for both ${\bf u}_h$  and the 
auxiliary variables $\lambda_h, s_h, {\bf q}_h$, including the $L^q$ error for ${\bf e}_h:={\bf u}-{\bf u}_h$ and  $ {\bf \eta}_h:=\mathcalQ_h\bu-\bu_h$, and 
the  errors $\3bar (e_\lambda, e_{\bm{q}})\3bar $ and $\3bar s_h \3bar$ defined in \eqref{norm} and \eqref{norm2}. We  test different values of $p>1$ with $p=2,3,4,5$.
 The right-hand side function, the boundary condition are calculated from the exact solution.  The coefficient $\epsilon_0$  in \eqref{iter:1} is taken as 
 $\epsilon=10^{-6/(p-1)}$,  $\rho_3=1$, 
 and $\rho_i, i=1,2$ are carefully taken according to different problems.  We stop our 
 iterative procedure when the maximum error
between the $m$-th step and the $(m+1)$-th step  reaches the accuracy $10^{-5}$.

\begin{example}
In this test, we consider the model problem \eqref{EQ:div-curl} in the domain
 $\Omega=(0,1)^3$ with $\varepsilon ={\rm diag}(3,2,1)$. The right-hand side function and the boundary condition are chosen such that 
 the exact solution  to this problem is 
 \[
{\bf u}(x,y,z)=\left( \begin{array}{cc}
   \sin(\pi x)\cos(\pi y)  \\
    -\sin(\pi y)\cos(\pi x) \\
    0 \\
 \end{array}
\right )+ \left( \begin{array}{cc}
   x \\
    y \\
  z \\
 \end{array}
\right ).
\]
  It is easy to see that ${\bf u} \in [H^1(\Omega)]^3$. 
  \end{example}
   The problem is solved by \eqref{iter:1} with the coefficients $\rho_1=\rho_2=1$ for $p=2$ and 
   $\rho_1=\rho_2=9\times10^{p-1}$ for $p\ge 2$. 
   Table \ref{1} illustrates the approximation error and the rate of convergence 
  for the primal variable ${\bf u}_h$ and the auxiliary variables  $\lambda_h, s_h, {\bf q}_h$ with $p=2,\ldots, 5$. 
  We observe a convergence rate of ${\mathcal O}(h)$ for the error $\|\varepsilon^{\frac{1}{q}}  {\bf e}_h\|_{L^q}$. For the dual variables $\lambda_h,{\bf q}_h$, 
  the table suggests a $p$-dependence rate of the convergence, i.e., ${\mathcal O}(h)$ for $p=2$,  ${\mathcal O}(h^{0.6})$ for $p=3$, ${\mathcal O}(h^{0.55})$ for $p=4$ and 
  ${\mathcal O}(h^{0.4})$ for $p=5$.  
 Note that the  convergence rate of $\3bar (e_\lambda, e_{\bm{q}})\3bar $ 
 is slightly higher than the
theoretical result $ {\mathcal  O}(h^{\frac qp})$ in \eqref{er1}.  As for the dual variable $s_h$,  we observe a better convergence rate than 
  the theoretical finding ${\mathcal O}(h)$ given in \eqref{er2}, which indicates a superconvergence result.

\begin{table}[h] 
\caption{Numerical error and rate of convergence for the $L^p$-PDWG method  for Example 1.}\label{1} \centering
\begin{threeparttable}
        \begin{tabular}{c |c c c c c c c  }
        \hline
   p & $1/h$&   $\|\varepsilon^{\frac{1}{q}}  {\bf e}_h\|_{L^q}$ & rate& $\3bar (e_\lambda, e_{\bm{q}})\3bar $ & rate& $\3bar s_h \3bar$ & rate   \\
       \hline \cline{2-8}

&2 & 1.52e-01 & --& 3.27e-02 & -- & 1.52e-03 & -- \\
2&4 & 7.67e-02 & 0.99 & 1.82e-02 & 0.84 & 3.05e-04 & 2.32 \\
&8 & 3.82e-02 & 1.00 & 9.37e-03 & 0.96 & 5.10e-05 & 2.58\\
&16 & 1.91e-02 & 1.00 & 4.72e-03 & 0.99 & 9.29e-06 & 2.46 \\
 \hline \cline{2-8}

&2 & 1.73e-01 & -- & 2.84e-01 & -- & 2.47e-03 & -- \\
3&4 & 8.65e-02 & 1.00 & 1.99e-01 & 0.52 & 3.20e-04 & 2.95 \\
&8 & 4.27e-02 & 1.02 & 1.30e-01 & 0.61& 3.78e-05 & 3.08 \\
&16 & 2.17e-02 & 0.98 & 8.01e-02 & 0.70 & 5.03e-06 & 2.91 \\

 \hline \cline{2-8}
& 2 & 1.97e-01 & -- & 4.26e-01 &--& 2.30e-03 & -- \\
4&4 & 1.02e-01 & 0.95 & 3.05e-01 & 0.48 & 2.31e-04 & 3.31 \\
&8 & 5.51e-02 & 0.88 & 2.07e-01 & 0.56 & 1.80e-05 & 3.68 \\
&16 & 2.92e-02 & 0.92 & 1.41e-01 & 0.55 & 1.59e-06 & 3.50 \\
  \hline \cline{2-8}
 & 2 & 2.13e-01 & -- & 4.16e-01 & -- & 5.79e-04 & -- \\
5&4 & 1.21e-01 & 0.82 & 3.02e-01 & 0.46 & 4.29e-05 & 3.76 \\
&8 & 6.38e-02 & 0.92& 2.25e-01 & 0.43& 2.89e-06 & 3.89 \\
&16 & 3.20e-02 & 1.00 & 1.70e-01 & 0.40 & 3.11e-07 & 3.22 \\
 \hline
 \end{tabular}
 \end{threeparttable}
\end{table}

\begin{example}\label{Example2} The domain in this test case is the $L$-shape domain, which is 
given by $\Omega=(0,1)^3\backslash\Omega_1$ with 
$\Omega_1=[0,1]\times [-1,0]\times [0,1]$.  We take the coefficient $\epsilon={\rm diag}(1,1,1)$ and 
the singular solution  in $[H^{2/3-\delta}(\Omega)]^3$: 
\[
   {\bf u}=\nabla \times (0,0, r^{2/3}\sin(\frac 23 \theta)). 
\]
Here $r=\sqrt{x^2+y^2}$ and $\theta={\rm arctan}(y/x)+c$ are the cylindrical coordinates.  We take  $c$  such that
${\bf u}\in  H(div)\cap  
H(curl)$. 
\end{example}

  We take the iterative scheme \eqref{iter:1} to solve this paper with the same coefficient choice of $\rho_1,\rho_2$ as Example 1.
   Numerical error and rate of convergence for the $L^p$-PDWG method are listed in  Table \ref{2}, from which 
   we observe an optimal convergence rate ${\mathcal O}(h^{\frac 23})$  for the errors
   $\|\varepsilon^{\frac{1}{q}}  {\bf e}_h\|_{L^q}$  and   $\3bar (e_\lambda, e_{\bm{q}})\3bar $ 
  for $p=2$. While,  as $p$ increases,  the convergence rate for ${\bf u}_h$ is improved 
 from ${\mathcal O}(h^{\frac 23})$ to ${\mathcal O}(h)$ (for $p=5$).  Like in Example 1,  
 the numerical convergence for the dual variables  is faster than the theory predicted in Theorem \ref{THM:ErrorEstimate4Triple}.

\begin{table}[h] 
\caption{Numerical error and rate of convergence for the $L^p$-PDWG method for Example 2.}\label{2} \centering
\begin{threeparttable}
        \begin{tabular}{c |c c c c c c c  }
        \hline
       p & $1/h$&   $\|\varepsilon^{\frac{1}{q}}  {\bf e}_h\|_{L^q}$ & rate& $\3bar (e_\lambda, e_{\bm{q}})\3bar $ & rate& $\3bar s_h \3bar$ & rate   \\
       \hline \cline{2-8}
       
&2 & 1.28e-01 & -- & 3.70e-02 & -- & 1.02e-03 & -- \\
2&4 & 8.25e-02 & 0.64 & 2.42e-02 & 0.61 & 3.18e-04 & 1.69 \\
&8 & 5.28e-02 & 0.64 & 1.55e-02 & 0.64& 9.66e-05 & 1.72 \\
&16 & 3.35e-02 & 0.65 & 9.92e-03 & 0.65 & 2.96e-05 & 1.71\\
 \hline \cline{2-8}
 
& 2 & 1.69e-01 & -- & 1.96e-01 & -- & 1.37e-04 & -- \\
3&4 & 9.82e-02 & 0.78& 1.44e-01 & 0.44 & 3.10e-05 & 2.15\\
&8 & 5.77e-02 & 0.77& 1.01e-01 & 0.52 & 6.29e-06 & 2.30 \\
&16 & 3.39e-02 & 0.77 & 6.66e-02 & 0.59 & 1.05e-06 & 2.58 \\

  \hline \cline{2-8}
 & 2 & 2.00e-01 & -- & 3.54e-01 & -- & 2.19e-04 & -- \\
4&4 & 1.18e-01 & 0.76 & 2.63e-01 & 0.43& 2.77e-05 & 2.98\\
&8 & 6.81e-02 & 0.79 & 1.67e-01 & 0.65& 1.79e-06 & 3.95 \\
&16 & 3.65e-02 & 0.90 & 9.79e-02 & 0.77 & 6.91e-08 & 4.70 \\

   \hline \cline{2-8}

&2 & 2.18e-01 & -- & 3.75e-01 & --& 5.47e-05 & -- \\
5&4 & 1.34e-01 & 0.70 & 2.30e-01 & 0.71 & 1.16e-06 & 5.56 \\
&8 & 7.12e-02 & 0.91 & 1.29e-01 & 0.83 & 1.15e-08 & 6.66 \\
&16 & 3.61e-02 & 0.98 & 7.15e-02 & 0.85 & 1.00e-10 & 6.84 \\
\hline
 \end{tabular}
 \end{threeparttable}
\end{table}

\begin{example}\label{Example3} 
In this example, we test  
a singular solution in the following vector potential form on a toroidal domain with $2$ holes: 
$\Omega=[(-1,\frac 32)]^2\backslash \{\Omega_1\cup\Omega_2\}$ with 
$\Omega_1=[-\frac 12,0]^2\times [0,\frac 12]$ and $\Omega_2=[\frac 12,1]\times[-\frac 12,0]\times [0,\frac 12]$. 
We take $\epsilon={\rm diag}(1,1,1)$ and 
\[
   {\bf u}=\nabla \times (0,0, r_1^{\gamma_1}\sin(2 \theta_1)+r_2^{\gamma_2}\sin(2 \theta_2)),
\]
where $(r_i,\theta_i), i=1,2$ are the   cylindrical coordinates centered at a nonconvex corner of the $i$-th hole. 
That is, 
\[
  r_1=\sqrt{x^2+y^2}, \ \theta_1={\rm arctan}(y/x)+c_1,\ \ r_2=\sqrt{(x-1)^2+y^2}, \ \theta_2={\rm arctan}(y/x)+c_2. 
 \]
 In our numerical experiments, we choose $
 \gamma_1 = 1/2, \gamma_2 = 2/3$ such that the vector field is singular near the nonconvex corners of both holes. 
 Note that ${\bf u}\in H^{\frac 12-\delta}$  
 in a neighborhood of the edge $\{x = 0,y = 0\}$, and 
 ${\bf u}\in H^{\frac 23-\delta}$  
 in a neighborhood of the edge $\{x = 1,y = 0\}$. 
 \end{example} 
 
From Table \ref{3},  we observe a convergence rate of  ${\mathcal O}(h^{\frac 12})$ for 
 the error $\|\varepsilon^{\frac{1}{q}}  {\bf e}_h\|_{L^q}$, and   ${\mathcal O}(h^{\frac 1p})$ 
 for the error $\3bar (e_\lambda, e_{\bm{q}})\3bar $  with 
 $p=2,\ldots 5$.   Again, it seems that the convergence rate for $\3bar s_h \3bar$  is better than the  theory 
  predicted in Theorem \ref{THM:ErrorEstimate4Triple}.

 \begin{table}[h] 
\caption{Numerical error and rate of convergence for the $L^p$-PDWG method for Example 3.}\label{3} \centering
\begin{threeparttable}
        \begin{tabular}{c |c c c c c c c  }
        \hline
     p & $1/h$&   $\|\varepsilon^{\frac{1}{q}}  {\bf e}_h\|_{L^q}$ & rate& $\3bar (e_\lambda, e_{\bm{q}})\3bar $ & rate& $\3bar s_h \3bar$ & rate   \\
       \hline \cline{2-8}
& 2 & 1.49e-00 & --& 3.48e-00 &-- & 5.16e-01 & -- \\
2&4 & 1.02e-00 & 0.55 & 2.60e-00 & 0.42 & 2.66e-01 & 0.96 \\
&8 & 6.84e-01 & 0.57 & 1.86e-00 & 0.48& 1.01e-01 & 1.39 \\
&16 & 4.69e-01 & 0.55 & 1.30e-00 & 0.51 & 3.49e-02 & 1.54\\
   \hline \cline{2-8}
  & 2 & 1.79e-00 & -- & 6.45e-01 & --& 3.75e-04 & -- \\
3&4 & 1.21e-00 & 0.57 & 5.14e-01 & 0.33 & 1.24e-04 & 1.60 \\
&8 & 8.12e-01 & 0.57 & 4.02e-01 & 0.35 & 4.14e-05 & 1.59 \\
&16 & 5.70e-01 & 0.51 & 3.12e-01 & 0.37 & 1.28e-05 & 1.69\\
  \hline \cline{2-8}   
&2 & 2.07e-00 & -- & 1.22e-00 & --& 5.38e-03 & -- \\
4&4 & 1.41e-00 & 0.55& 1.03e-00 & 0.25 & 1.50e-03 & 1.84\\
&8 & 9.77e-01 & 0.53 & 8.60e-01 & 0.26& 4.30e-04 & 1.80 \\
&16 & 6.78e-01 & 0.53 & 7.17e-01 & 0.26& 1.24e-04 & 1.80 \\

 \hline \cline{2-8}   
&2 & 2.18e-00 & -- & 9.54e-01 &-- & 9.98e-04 & -- \\
5&4 & 1.57e-00 & 0.47& 8.30e-01 & 0.20 & 2.77e-04 & 1.82 \\
&8 & 1.06e-00 & 0.56& 7.19e-01 & 0.21& 7.67e-05 & 1.85\\
&16 & 7.33e-01 & 0.53 & 6.22e-01 & 0.21 & 2.12e-05 & 1.86 \\
\hline
 \end{tabular}
 \end{threeparttable}
\end{table}

\begin{example} \label{example4}
In this test, we consider
a singular solution in the following vector potential form on a toroidal domain with $1$ holes: 
$\Omega=[(-1,\frac 12)]^2\times [0,\frac 12]\backslash \Omega_1$ with 
$\Omega_1=[-\frac 12,0]\times [0,\frac 12]^2$. 
We take $\epsilon={\rm diag}(1,1,1)$ and 
\[
   {\bf u}=\nabla \times (0,0, r^{\gamma}\sin(2 \theta)). 
\]
  We consider three cases: $\gamma=5/4, 1, 2/3$, where the regularity of the exact solution ranges  from smooth to singular. 
\end{example}
   The coefficients $\rho_1,\rho_2$ in \eqref{iter:1} are chosen as following:  $\rho_1=\rho_2=1$ for $p=2$, $\rho_1=\rho_2=3\times 10^3$  for $p=3$ and 
   $\rho_1=\rho_2=3\times 10^4$  for $p\ge 4$. Numerical error and rate of convergence for the $L^p$-PDWG method with different $\gamma$ are listed in 
   Tables \ref{4}-\ref{6}.  We observe the following result: 
   \begin{itemize} 
    \item Regular solution, i.e., $\gamma=5/4$, $\|\varepsilon^{\frac{1}{q}}  {\bf e}_h\|_{L^q}$  has optimal convergence rate 
   ${\mathcal O}(h)$ for all $p\ge 2$,  $\3bar (e_\lambda, e_{\bm{q}})\3bar $ has optimal convergence rate 
   ${\mathcal O}(h)$ for $p=2$,  and $p$-dependency rate when $p\ge 3$, which is slightly better than the theoretical rate ${\mathcal O}(h^{\frac qp})$. 
   As for $s_h$, a superconvergent order still observed in this cases.  
   \item Singular case $\gamma=1$, where ${\bf u}\in H^{1-\delta}$ and ${\bf u}\notin H(curl)$. 
   An asymptotical rate of ${\mathcal O}(h^{0.9}$ is observed for the error 
  $\|\varepsilon^{\frac{1}{q}}  {\bf e}_h\|_{L^q}$   with all $p\ge 2$.  The convergence behavior for the auxiliary variable are the same as that for the 
   regular case $\gamma=\frac 54$. 
   \item Singular case $\gamma=2/3$, where ${\bf u}\in H^{\frac 23-\delta}$ and ${\bf u}\notin H(curl)$.   
   $\|\varepsilon^{\frac{1}{q}}  {\bf e}_h\|_{L^q}$ has an optimal convergence rate 
   ${\mathcal O}(h^{\frac 23})$ for all $p\ge 2$, 
   and $\3bar (e_\lambda, e_{\bm{q}})\3bar $ shows a rate of ${\mathcal O}(h^{0.6})$ for $p=2$ and
   ${\mathcal O}(h^{\frac q p})$ for 
     all $p\ge 3$. 
   \end{itemize}

\begin{table}[h] 
\caption{Numerical error and rate of convergence for the $L^p$-PDWG method with $\gamma=5/4$ for Example 4.}\label{4} \centering
\begin{threeparttable}
        \begin{tabular}{c |c c c c c c c  }
        \hline
     p & $1/h$&   $\|\varepsilon^{\frac{1}{q}}  {\bf e}_h\|_{L^q}$ & rate& $\3bar (e_\lambda, e_{\bm{q}})\3bar $ & rate& $\3bar s_h \3bar$ & rate   \\
       \hline \cline{2-8}
&2 & 3.96e-01 &-- & 9.07e-01 & --& 6.51e-02 & -- \\
2&4 & 2.07e-01 & 0.93 & 5.01e-01 & 0.86 & 3.23e-02 & 1.01 \\
&8 & 1.06e-01 & 0.97 & 2.67e-01 & 0.91 & 9.48e-03 & 1.77 \\
&16 & 5.38e-02 & 0.98 & 1.39e-01 & 0.95 & 2.17e-03 & 2.12 \\
   \hline \cline{2-8}
&2 & 4.61e-01 & -- & 4.14e-01 & -- & 4.17e-04 & -- \\
3&4 & 2.39e-01 & 0.95 & 2.88e-01 & 0.53 & 1.05e-04 & 1.99 \\
&8 & 1.27e-01 & 0.91& 1.94e-01 & 0.57 & 2.20e-05 & 2.25 \\
&16 & 6.65e-02 & 0.94 & 1.25e-01 & 0.63 & 4.23e-06 & 2.38 \\

 \hline \cline{2-8}
&2 & 4.92e-01 & -- & 5.92e-01 & -- & 5.23e-04 & -- \\
4&4 & 2.67e-01 & 0.88 & 4.29e-01 & 0.46 & 1.05e-04 & 2.31 \\
&8 & 1.41e-01 & 0.92& 3.04e-01 & 0.50 & 1.78e-05 & 2.57 \\
&16 & 7.23e-02 & 0.97& 2.16e-01 & 0.50 & 2.74e-06 & 2.70 \\
 \hline \cline{2-8}
 &2 & 5.33e-01 & --& 8.43e-01 & -- & 3.56e-03 &-- \\
5&4 & 2.93e-01 & 0.86 & 6.50e-01 & 0.38 & 6.15e-04 & 2.53 \\
&8 & 1.50e-01 & 0.97 & 4.98e-01 & 0.39 & 1.04e-04 & 2.56 \\
&16 & 7.55e-02 & 0.99 & 3.80e-01 & 0.39 & 1.54e-05 & 2.75 \\
 \hline
 \end{tabular}
 \end{threeparttable}
\end{table}

\begin{table}[h] 
\caption{Numerical error and rate of convergence for the $L^p$-PDWG method with $\gamma=1$ for Example 4.}\label{5} \centering
\begin{threeparttable}
        \begin{tabular}{c |c c c c c c c  }
        \hline
     p & $1/h$&   $\|\varepsilon^{\frac{1}{q}} {\bf e}_h\|_{L^q}$ & rate& $\3bar (e_\lambda, e_{\bm{q}})\3bar $ & rate& $\3bar s_h \3bar$ & rate   \\
       \hline \cline{2-8}
&2 & 5.33e-01 & --& 1.22e-00 & -- & 1.14e-01 & -- \\
2&4 & 3.01e-01 & 0.83 & 7.27e-01 & 0.74 & 5.65e-02 & 1.01 \\
&8 & 1.63e-01 & 0.88& 4.15e-01 & 0.81 & 1.79e-02 & 1.66 \\
&16 & 8.86e-02 & 0.88 & 2.30e-01 & 0.85 & 4.62e-03 & 1.96 \\
 \hline \cline{2-8}
&2 & 6.02e-01 & --& 5.84e-01 & -- & 1.22e-03 & -- \\
3&4 & 3.29e-01 & 0.87 & 3.98e-01 & 0.55 & 2.55e-04 & 2.26 \\
&8 & 1.87e-01 & 0.82 & 2.78e-01 & 0.52& 5.63e-05 & 2.18 \\
&16 & 1.03e-01 & 0.86 & 1.88e-01 & 0.56 & 1.22e-05 & 2.21 \\
 \hline \cline{2-8}
 &2 & 6.50e-01 & -- & 7.57e-01 & -- & 1.26e-03 & -- \\
4&4 & 3.80e-01 & 0.77 & 5.68e-01 & 0.41 & 2.52e-04 & 2.32 \\
&8 & 2.14e-01 & 0.83 & 4.19e-01 & 0.44 & 4.91e-05 & 2.36 \\
&16 & 1.13e-01 & 0.93 & 3.07e-01 & 0.45& 9.16e-06 & 2.42 \\
 \hline \cline{2-8}
 &2 & 6.88e-01 & -- & 8.31e-01 & --& 1.60e-03 & -- \\
5&4 & 4.14e-01 & 0.73 & 6.57e-01 & 0.34 & 3.09e-04 & 2.38 \\
&8 & 2.26e-01 & 0.87 & 5.16e-01 & 0.35 & 5.79e-05 & 2.41 \\
&16 & 1.19e-01 & 0.93 & 4.04e-01 & 0.35 & 1.01e-05 & 2.52 \\
 \hline
 \end{tabular}
 \end{threeparttable}
\end{table}

\begin{table}[h] 
\caption{Numerical error and rate of convergence for the $L^p$-PDWG method with $\gamma=2/3$ for Example 4.}\label{6} \centering
\begin{threeparttable}
        \begin{tabular}{c |c c c c c c c  }
         \hline
      p & $1/h$&   $\|\varepsilon^{\frac{1}{q}}  {\bf e}_h\|_{L^q}$ & rate& $\3bar (e_\lambda, e_{\bm{q}})\3bar $ & rate& $\3bar s_h \3bar$ & rate   \\
       \hline \cline{2-8}
&2 & 8.87e-01 & -- & 2.09e-00 & -- & 2.77e-01 & -- \\
2&4 & 5.77e-01 & 0.62 & 1.45e-00 & 0.52 & 1.38e-01 & 1.01 \\
&8 & 3.62e-01 & 0.68 & 9.67e-01 & 0.59 & 4.92e-02 & 1.49 \\
&16 & 2.29e-01 & 0.66& 6.26e-01 & 0.63& 1.55e-02 & 1.67 \\
 \hline \cline{2-8}
 
 &2 & 1.02e-00 & --& 8.04e-01 & -- & 2.07e-03 & -- \\
3&4 & 6.16e-01 & 0.73 & 5.77e-01 & 0.48 & 4.91e-04 & 2.08 \\
&8 & 3.94e-01 & 0.64 & 4.37e-01 & 0.40 & 1.39e-04 & 1.82 \\
&16 & 2.47e-01 & 0.67 & 3.28e-01 & 0.42 & 4.21e-05 & 1.73 \\

 \hline \cline{2-8}
&2 & 1.07e-00 & -- & 1.06e-00 & -- & 4.46e-03 & -- \\
4&4 & 7.10e-01 & 0.59 & 8.64e-01 & 0.30 & 1.13e-03 & 1.98 \\
&8 & 4.63e-01 & 0.62& 6.93e-01 & 0.32 & 2.86e-04 & 1.98 \\
&16 & 2.90e-01 & 0.67 & 5.53e-01 & 0.32 & 7.19e-05 & 1.99 \\

\hline \cline{2-8}

&2 & 1.11e-00 & -- & 1.09e-00 &-- & 5.79e-03 & -- \\
5&4 & 7.73e-01 & 0.52 & 9.19e-01 & 0.25 & 1.41e-03 & 2.04 \\
&8 & 4.94e-01 & 0.65& 7.71e-01 & 0.25 & 3.44e-04 & 2.03 \\
&16 & 3.09e-01 & 0.68 & 6.45e-01 & 0.26 & 8.25e-05 & 2.06 \\

 \hline
 \end{tabular}
 \end{threeparttable}
\end{table}

 \begin{example}\label{example5} 
In this test, we reveal some computational results for 
a test problem where the existence of  a harmonic vector field has effect on convergence rate. 
We consider the problem 
 on a toroidal  domain with $ 2$ holes,  which is the same as that in Example \ref{Example3}. 
We take $\epsilon={\rm diag}(1,1,1)$ and 
\[
   {\bf u}=\nabla \times (0,0, r_1^{\gamma_1}\sin( \theta_1)+r_2^{\gamma_2}\sin( \theta_2))+\beta  \left( \begin{array}{cc}
   e^y\sin(z) \\
    e^x\sin(z) \\
  z \\
 \end{array}
\right )
\]  
  with $(\gamma_1,\gamma_2,\beta)=(\frac 45, \frac 23,\frac 1{40})$. 
\end{example}
   The coefficients $\rho_1,\rho_2$ in \eqref{iter:1} are chosen as following:  $\rho_1=\rho_2=1$ for $p=2$, $\rho_1=\rho_2= 5\times 10^4$  for $p\ge 3$.  
   The plot of the vector field ${\bf\eta}_h$ is provided in Figure \ref{fig:1}, and the 
  errors and rates of convergence of the $L^p$-PDWG method for the primal variable and the dual variables are given in  Table \ref{10}. 
As indicated by Theorem \ref{THM:ErrorEstimate4uh},  the numerical solution ${\bf u}_h$ approximates the exact solution ${\bf u}$, up to a harmonic field. 
As we may observe from Figure \ref{fig:1}, the vector field ${\bf \eta}_h$  is an approximate harmonic field with normal boundary condition.
Furthermore, due to the presence of the harmonic field vector, the error ${\bf \eta}_h$ or ${\bf e}_h$ may not exhibit a convergence. Our numerical result in 
Table \ref{10} verifies this point.  We do not observe  any convergence for the vector field ${\bf u}$.  It is noteworthy that although 
 the vector field ${\bf u}_h$ is not convergent to  ${\bf u}$ while our iterative algorithm is still convergent. We list in Table \ref{10} that the iterative number 
 used in the iterative procedure.  As for the dual variable $\lambda_h, {\bf q}_h$,  we observe a rate of ${\mathcal O}(h^{\frac qp})$ for $p\ge 3$.
  The numerical performance is in consistency with our theory as established in Theorem  \ref{THM:ErrorEstimate4Triple} 
   for the convergence of  $e_{\lambda}, e_{\bf q}$.  The convergence rate for $s_h$ is still better than the one given in \eqref{er2}.

\begin{figure}[htb]
  \centering
  \includegraphics[width=.32\textwidth]{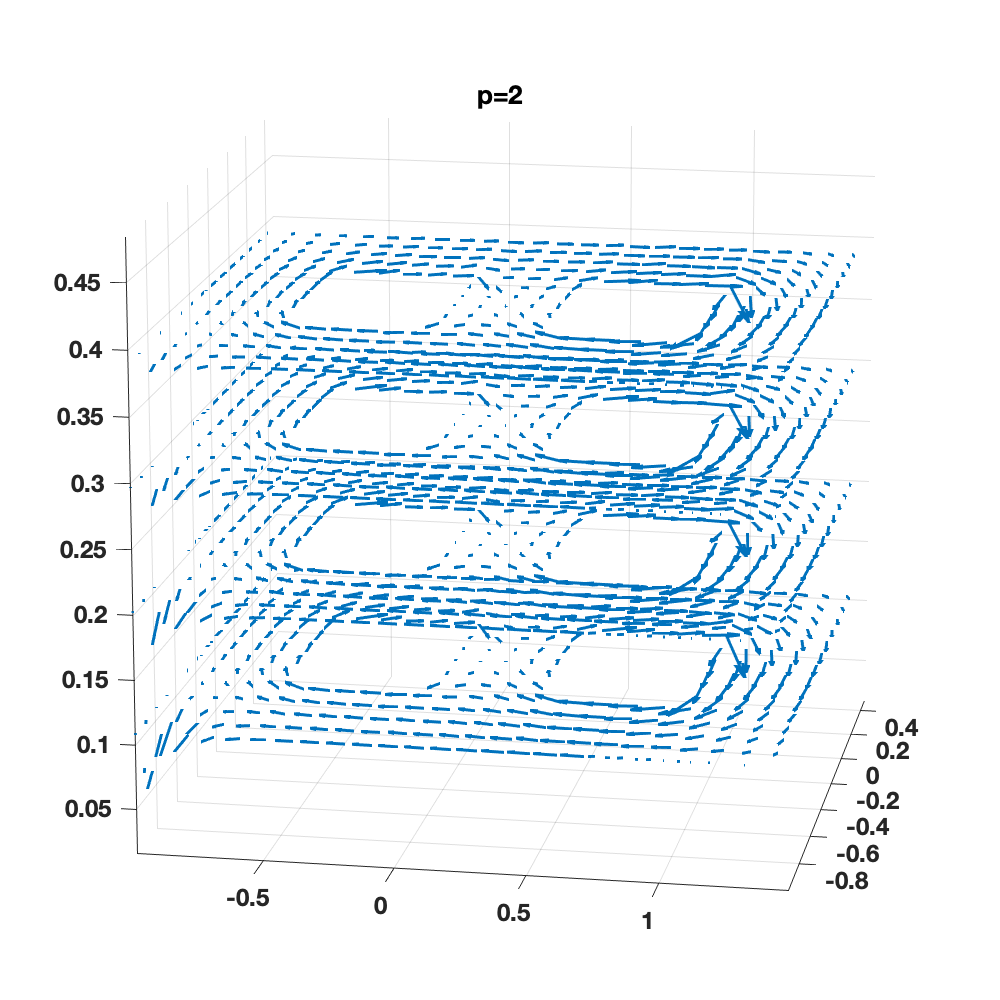}
  \includegraphics[width=.32\textwidth]{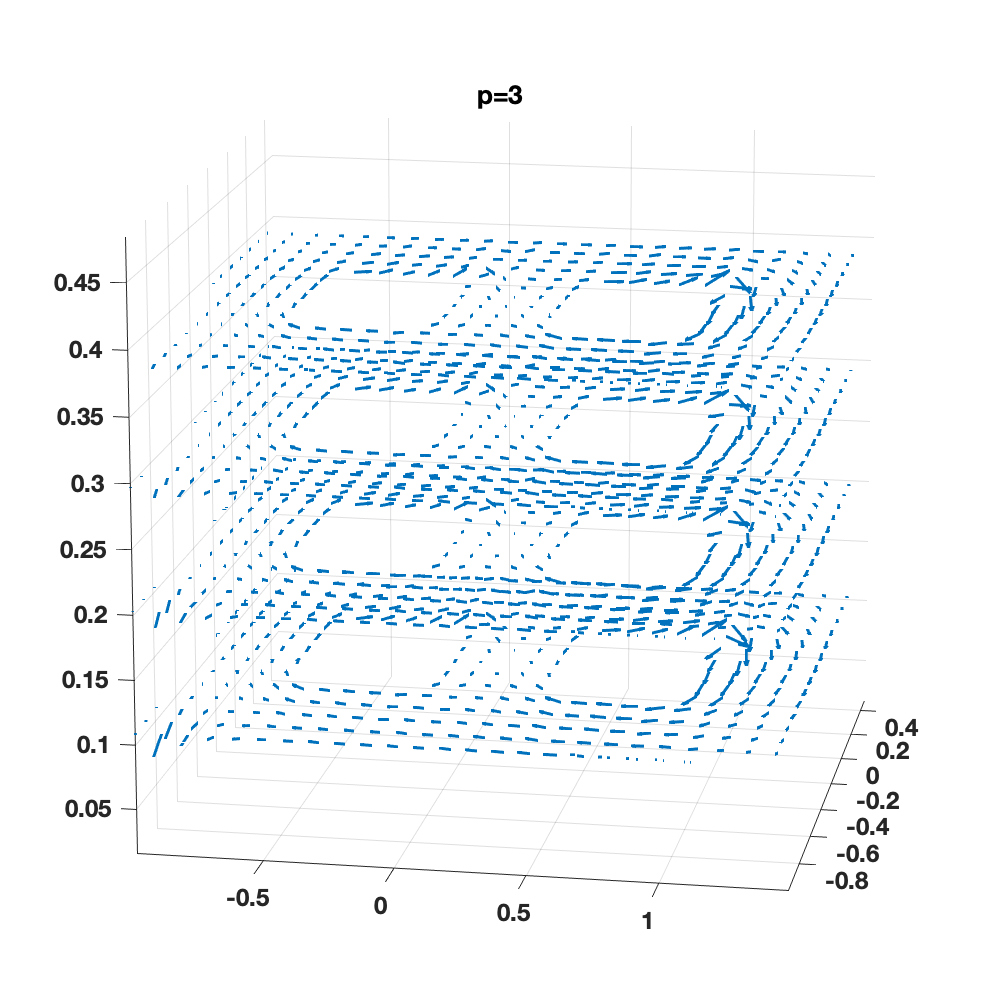}\\
   \includegraphics[width=.32\textwidth]{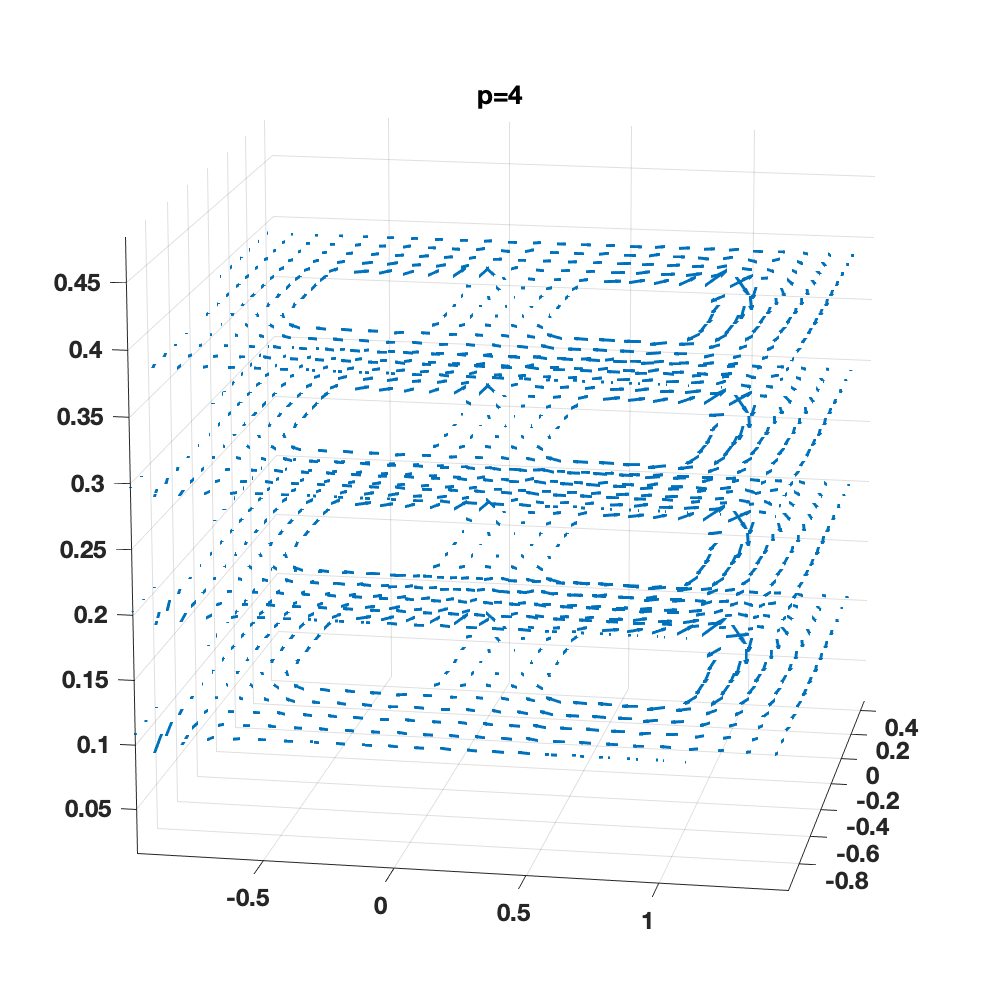}
   \includegraphics[width=.32\textwidth]{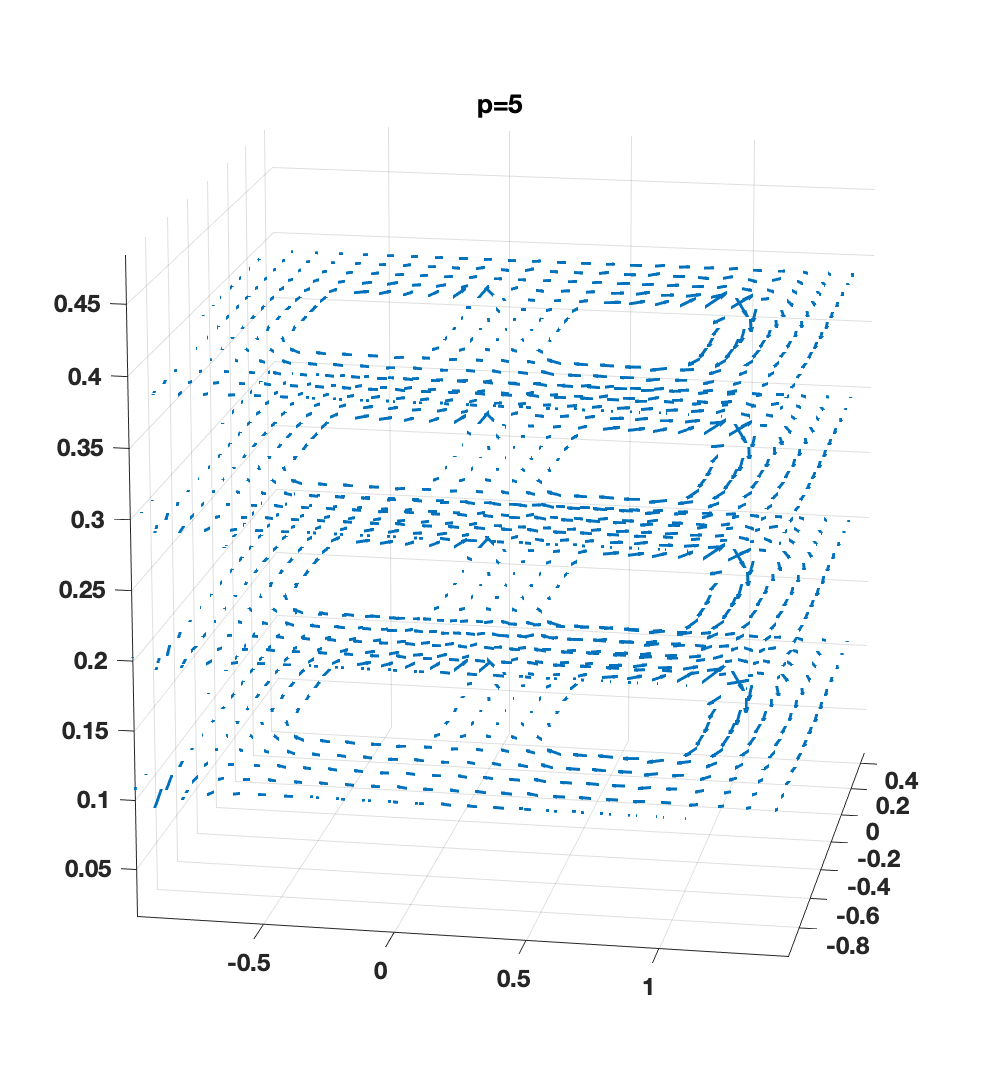}
 \caption{Plot of the  vector fields ${\bf \eta}_h$ calculated by the $L^p$-PDWG method of Example 5 with different values of $p$.}
  \label{fig:1}
\end{figure}

\begin{table}[h] 
\caption{Numerical error and rate of convergence for the $L^p$-PDWG method for Example 5.}\label{10} \centering
\begin{threeparttable}
        \begin{tabular}{c |c c c c c c c cc c }
         \hline
    p & $1/h$&   $\|\varepsilon^{\frac{1}{q}}  {\bf e}_h\|_{L^q}$ & rate& $\|\varepsilon^{\frac{1}{q}}  {\bf \eta}_h\|_{L^q}$ &rate& $\3bar (e_\lambda, e_{\bm{q}})\3bar $ & rate& $\3bar s_h \3bar$ & rate & It.   \\
       \hline \cline{2-11}
&2 & 2.60e-01 & -- & 1.90e-01 & -- & 3.79e-01 & -- & 5.83e-02 & -- & 1 \\
2&4 & 2.03e-01 & 0.36 & 1.69e-01 & 0.17 & 2.55e-01 & 0.57 & 2.93e-02 & 0.99 & 1 \\
&8 & 1.70e-01 & 0.26 & 1.54e-01 & 0.13 & 1.67e-01 & 0.61 & 1.03e-02 & 1.51 & 1 \\
&16 & 1.53e-01 & 0.15 & 1.46e-01 & 0.08& 1.07e-01 & 0.64 & 3.26e-03 & 1.66 & 1 \\
\hline \cline{2-11}
&2 & 2.78e-01 & -- & 2.16e-01 & -- & 1.35e-01 & -- & 1.52e-05 & -- & 14 \\
3&4 & 2.28e-01 & 0.29 & 2.02e-01 & 0.09 & 9.93e-02 & 0.44 & 4.98e-06 & 1.60 & 15 \\
&8 & 1.98e-01 & 0.20 & 1.88e-01 & 0.11 & 7.11e-02 & 0.48 & 1.20e-06 & 2.06 & 18 \\
&16 & 1.80e-01 & 0.14 & 1.76e-01 & 0.09 & 5.06e-02 & 0.49 & 2.88e-07 & 2.06 & 19 \\
\hline \cline{2-11}
& 2 & 3.12e-01 & -- & 2.55e-01 & -- & 3.98e-01 & -- & 1.60e-04 & 0 & 18 \\
4&4 & 2.62e-01 & 0.25 & 2.40e-01 & 0.09 & 3.12e-01 & 0.35 & 4.54e-05 & 1.82 & 28 \\
&8 & 2.25e-01 & 0.22& 2.15e-01 & 0.15 & 2.45e-01 & 0.35 & 1.08e-05 & 2.07 & 28 \\
&16 & 1.99e-01 & 0.18 & 1.95e-01 & 0.14 & 1.93e-01 & 0.34 & 2.60e-06 & 2.05 & 26 \\
\hline \cline{2-11}
&2 & 3.42e-01 & -- & 2.87e-01 & --& 6.12e-01 & -- & 8.62e-04 & -- & 29 \\
5&4 & 2.84e-01 & 0.27 & 2.63e-01 & 0.13 & 5.05e-01 & 0.28 & 2.32e-04 & 1.89 & 23 \\
&8 & 2.37e-01 & 0.26 & 2.28e-01 & 0.20 & 4.19e-01 & 0.27& 5.59e-05 & 2.06 & 24 \\
&16 & 2.08e-01 & 0.19 & 2.04e-01 & 0.16 & 3.48e-01 & 0.27 & 1.31e-05 & 2.09 & 30 \\
  \hline
 \end{tabular}
 \end{threeparttable}
\end{table}

\begin{example}
In this test, 
we consider the problem 
 on a toroidal  domain with $ 1$ holes,  which is the same as that in Example \ref{example4}. 
We take $\epsilon={\rm diag}(1,1,1)$ and 
\[
   {\bf u}=\nabla \times (0,0, r^{\gamma}\sin( \theta)+\beta  \left( \begin{array}{cc}
 \sin(\pi x)\cos(\pi y)\sin(\pi z) \\
   \cos(\pi x)\sin(\pi y)\sin(\pi z) \\
  0 \\
 \end{array}
\right )
\]  
  with $(\gamma,\beta)=( \frac 23,\frac 1 8)$. 
\end{example}
We take the iterative parameters as those in Example \ref{example5}, and plot the 
 the vector field $\eta_h$ in Figure \ref{fig:1}, and  present  in Table \ref{11} the
  errors and rates of convergence for the primal variable and the dual variables approximation.  
  Just the same as that in Example \ref{example5},  the numerical results do not demonstrate any convergence for the vector field ${\bf u}$, 
  while show a  rate of ${\mathcal O}(h^{\frac qp})$ for $\3bar (e_\lambda, e_{\bm{q}})\3bar $ and a supercovnergence rate  ${\mathcal O}(h^2)$ for 
  $\3bar s_h \3bar$ when $p\ge 3$.  Again we observe that the iterative scheme is still convergent and 
  the vector field ${\bf\eta}_h$  is an approximate harmonic field with normal boundary condition.

\begin{figure}[htb]
  \centering
  \includegraphics[width=.32\textwidth]{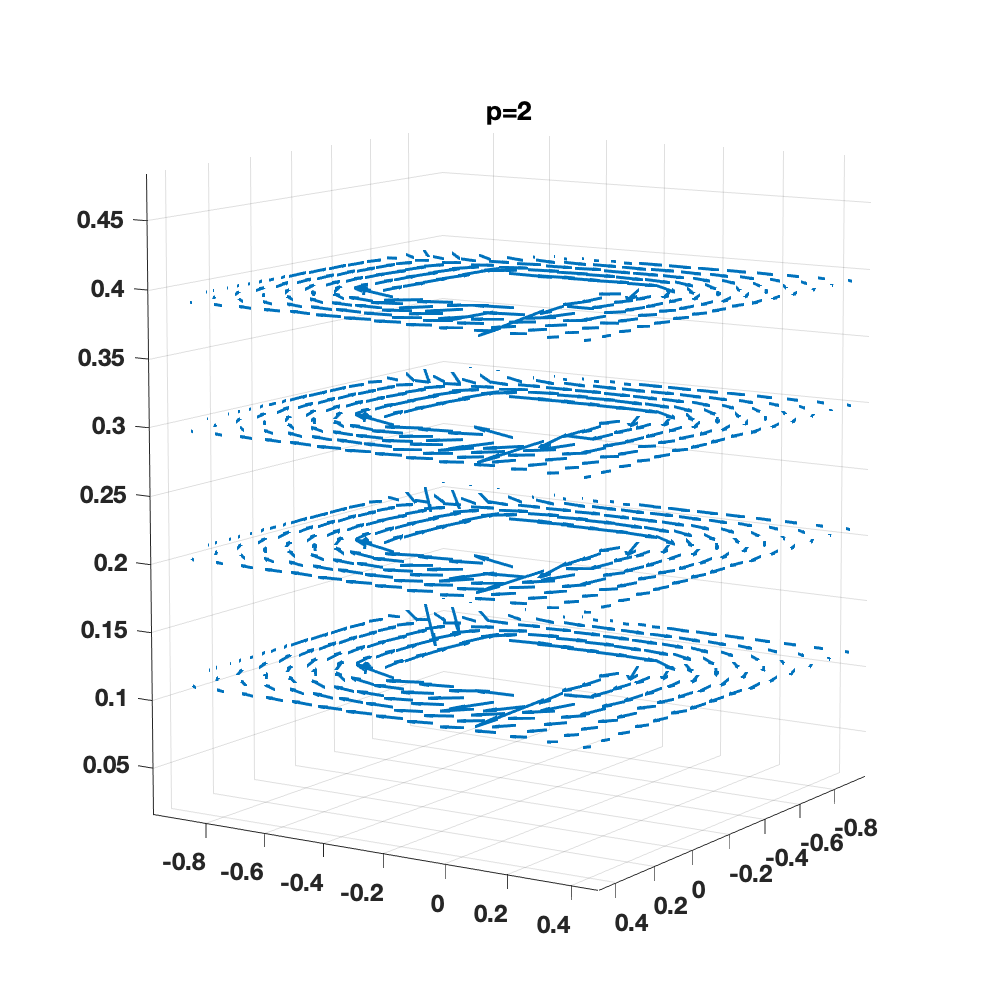}
  \includegraphics[width=.32\textwidth]{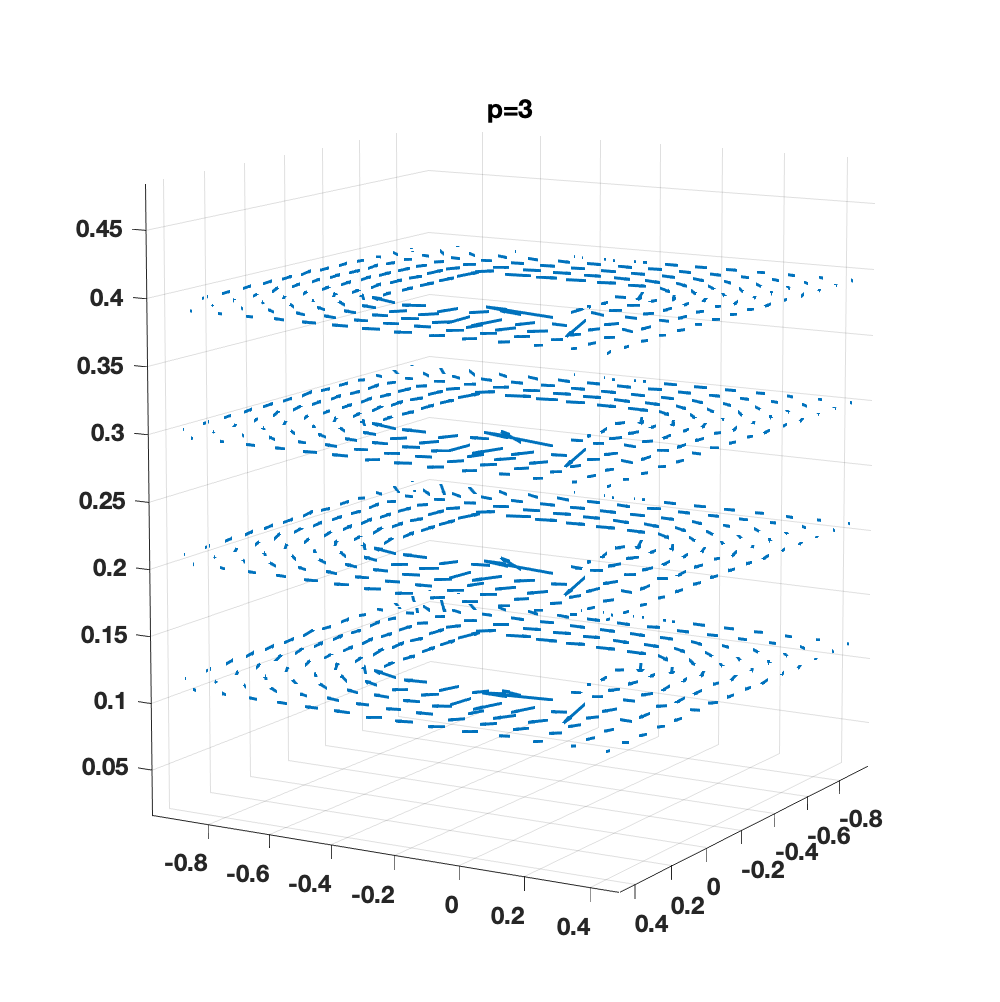}\\
   \includegraphics[width=.32\textwidth]{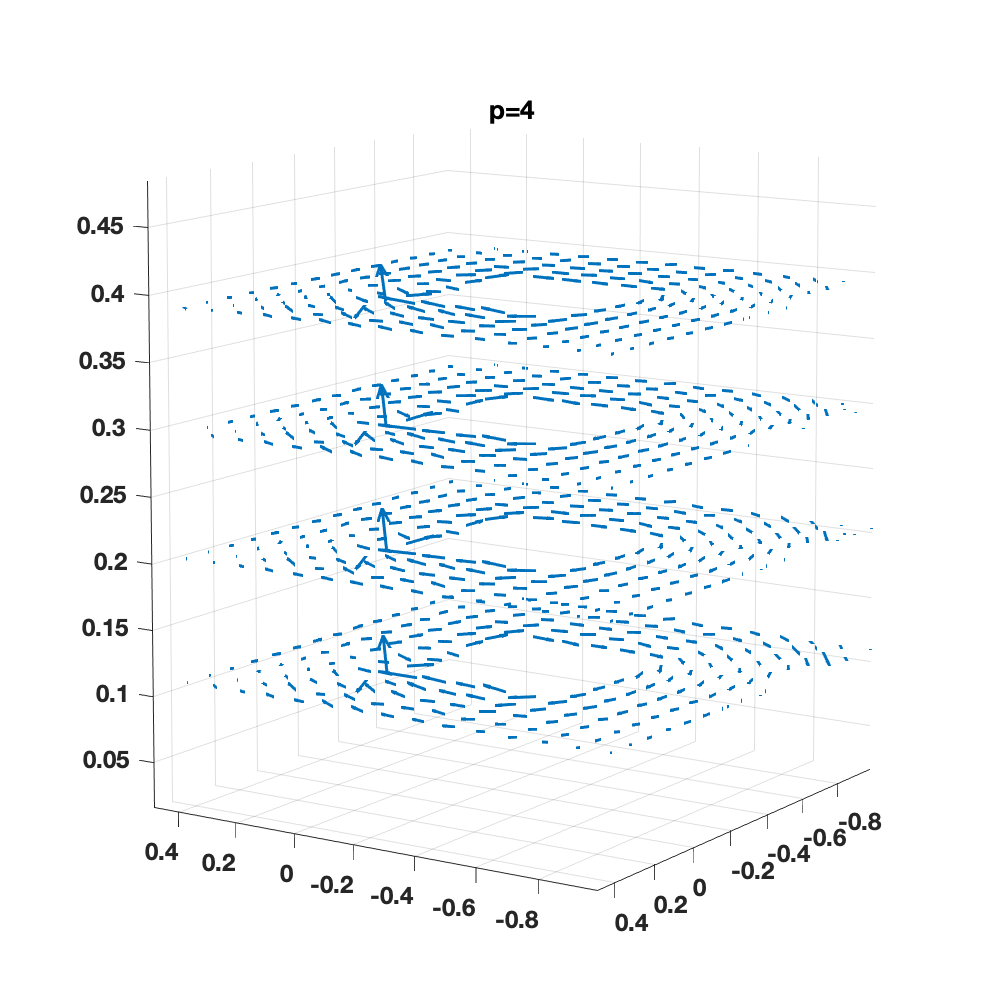}
   \includegraphics[width=.32\textwidth]{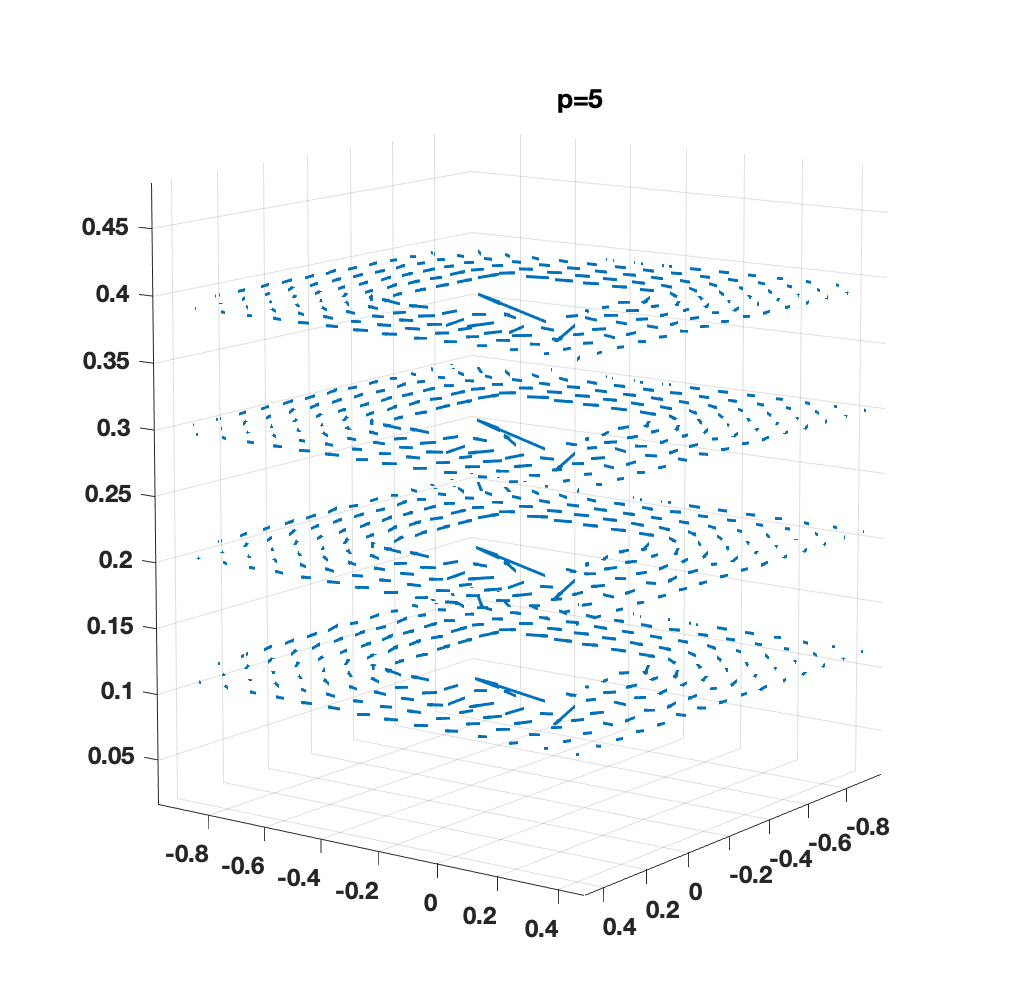}
 \caption{Plot of the  vector fields ${\bf \eta}_h$ calculated by the $L^p$-PDWG method of Example 5 with different values of $p$.}
  \label{fig:14}
\end{figure}

\begin{table}[h] 
\caption{Numerical error and rate of convergence for the $L^p$-PDWG method  for Example 6.}\label{11} \centering
\begin{threeparttable}
        \begin{tabular}{c |c c c c c c c cc c }
         \hline
    p & $1/h$&   $\|\varepsilon^{\frac{1}{q}}  {\bf e}_h\|_{L^q}$ & rate& $\|\varepsilon^{\frac{1}{q}}  {\bf \eta}_h\|_{L^q}$ &rate& $\3bar (e_\lambda, e_{\bm{q}})\3bar $ & rate& $\3bar s_h \3bar$ & rate & It.   \\
       \hline \cline{2-11}
&2 & 2.18e-01 & -- & 1.61e-01 & -- & 2.98e-01 & --& 5.18e-02 &-- & 1 \\
2&4 & 1.75e-01 & 0.32 & 1.45e-01 & 0.15 & 2.16e-01 & 0.46 & 2.67e-02 & 0.95 & 1 \\
&8 & 1.47e-01 & 0.25 & 1.33e-01 & 0.13 & 1.47e-01 & 0.56 & 9.49e-03 & 1.50 & 1 \\
&16 & 1.33e-01 & 0.15 & 1.26e-01 & 0.07& 9.61e-02 & 0.61 & 3.04e-03 & 1.64 & 1 \\
 \hline \cline{2-11}

&2 & 2.14e-01 & -- & 1.67e-01 & -- & 1.14e-01 & -- & 1.19e-05 & --& 14 \\
3&4 & 1.81e-01 & 0.24 & 1.60e-01 & 0.06& 8.85e-02 & 0.37 & 4.29e-06 & 1.48 & 16 \\
&8 & 1.58e-01 & 0.19 & 1.49e-01 & 0.10 & 6.51e-02 & 0.44 & 1.05e-06 & 2.03 & 18 \\
&16 & 1.45e-01 & 0.13& 1.41e-01 & 0.08& 4.72e-02 & 0.47 & 2.58e-07 & 2.03 & 19 \\

 \hline \cline{2-11}

&2 & 2.28e-01 & -- & 1.85e-01 & -- & 3.56e-01 & -- & 1.27e-04 & -- & 18 \\
4&4 & 1.99e-01 & 0.20& 1.81e-01 & 0.03 & 2.89e-01 & 0.30& 3.89e-05 & 1.7 & 30 \\
&8 & 1.73e-01 & 0.20 & 1.66e-01 & 0.13 & 2.30e-01 & 0.33 & 9.44e-06 & 2.04 & 28 \\
&16 & 1.54e-01 & 0.17 & 1.51e-01 & 0.13 & 1.84e-01 & 0.33 & 2.33e-06 & 2.02 & 26 \\

 \hline \cline{2-11}

&2 & 2.42e-01 & -- & 2.01e-01 & -- & 5.63e-01 & -- & 7.06e-04 & -- & 18 \\
5&4 & 2.11e-01 & 0.20 & 1.94e-01 & 0.05 & 4.75e-01 & 0.24 & 1.99e-04 & 1.83 & 24 \\
&8 & 1.79e-01 & 0.23 & 1.72e-01 & 0.17 & 3.99e-01 & 0.25 & 4.86e-05 & 2.03 & 24 \\
&16 & 1.58e-01 & 0.18 & 1.55e-01 & 0.15 & 3.35e-01 & 0.26 & 1.16e-05 & 2.06 & 31 \\
\hline
 \end{tabular}
 \end{threeparttable}
\end{table}

\end{document}